\documentclass[12pt,reqno,a4paper]{amsart}
\usepackage{fullpage}
\usepackage[utf8]{inputenc}
\usepackage[T1]{fontenc}
\usepackage{hyperref} 
\usepackage{enumitem}
\usepackage{amsmath,amssymb,amsthm,color}
\usepackage{mathtools}
\usepackage{todonotes,framed}
\usepackage{nameref}
\usepackage{enumitem}
\usepackage{graphicx}
\usepackage{pgfplots}
\usepackage{pgfplotstable}

\usepackage{fancyhdr}
\advance\footskip0.4cm
\advance\textheight-0.4cm
\calclayout
\newcommand*\patchAmsMathEnvironmentForLineno[1]{%
  \expandafter\let\csname old#1\expandafter\endcsname\csname #1\endcsname
  \expandafter\let\csname oldend#1\expandafter\endcsname\csname end#1\endcsname
  \renewenvironment{#1}%
     {\linenomath\csname old#1\endcsname}%
     {\csname oldend#1\endcsname\endlinenomath}}% 
\newcommand*\patchBothAmsMathEnvironmentsForLineno[1]{%
  \patchAmsMathEnvironmentForLineno{#1}%
  \patchAmsMathEnvironmentForLineno{#1*}}%
\AtBeginDocument{%
\patchBothAmsMathEnvironmentsForLineno{equation}%
\patchBothAmsMathEnvironmentsForLineno{align}%
\patchBothAmsMathEnvironmentsForLineno{flalign}%
\patchBothAmsMathEnvironmentsForLineno{alignat}%
\patchBothAmsMathEnvironmentsForLineno{gather}%
\patchBothAmsMathEnvironmentsForLineno{multline}%
}
\usepackage[mathlines]{lineno}

\makeatletter
\def\@seccntformat#1{\hspace*{4mm}%
  \protect\textup{\protect\@secnumfont
    \ifnum\pdfstrcmp{subsection}{#1}=0 \bfseries\fi% subsection # in \bfseries
    \csname the#1\endcsname
    \protect\@secnumpunct
  }%
}
\makeatother

\makeatletter
\def\section{\@startsection{section}{1}%
\z@{.7\linespacing\@plus\linespacing}{.5\linespacing}%
{\normalsize\scshape\bfseries\centering}}

\renewcommand{\@secnumfont}{\bfseries}
\makeatother

\def\A{\boldsymbol{A}}

\def\N{\mathbb{N}}
\def\R{\mathbb{R}}
\def\T{\mathbb{T}}

\def\MM{\mathcal{M}}
\def\OO{\mathcal{O}}
\def\QQ{\mathcal{Q}}

\def\TT{\mathcal{T}}
\def\UU{\mathcal{U}}
\def\XX{\mathcal{X}}

\def\f{\boldsymbol{f}}
\def\g{\boldsymbol{g}}
\def\k{{\underline{k}}}
\def\m{{\underline{m}}}
\def\n{{\underline{n}}}
\def\coarse{H}
\def\fine{h}

\newcommand{\norm}[3][]{#1\| #2 \|_{#3}}
\newcommand{\enorm}[2][]{#1|\!#1|\!#1| #2 #1|\!#1|\!#1|}

\newcommand{\set}[3][\big]{#1\{#2 \,:\, #3 #1\}}
\newcommand{\jump}[1]{[\![#1]\!]}
\newcommand{\normalvec}{\boldsymbol{n}}
\def\d#1{\,{\rm d}#1}
\def\div{{\rm div}\,}
\def\refine{{\tt refine}}
\def\work{{\tt work}}
\def\reff#1#2{\stackrel{\eqref{#1}}{#2}}
\def\refp#1#2{\stackrel{\phantom{\eqref{#1}}}{#2}}
\renewcommand{\vec}[1]{\boldsymbol{#1}}

\def\qctr{q_{\rm ctr}}
\def\Cstab{C_{\rm stab}}
\def\Crel{C_{\rm rel}}
\def\Cdrel{C_{\rm drel}}
\def\qred{q_{\rm red}}
\def\Clin{C_{\rm lin}}
\def\qlin{q_{\rm lin}}

\def\Cmark{C_{\rm mark}}
\def\Copt{C_{\rm opt}}
\def\lctr{\lambda_{\rm ctr}}
\def\qaux{q_{\rm aux}}
\def\Caux{C_{\rm aux}}
\def\lconv{\lambda_\star}

\def\Cgoal{C_{\rm goal}}
\def\Ccls{C_{\rm cls}}
\def\Cell{C_{\rm ell}}
\def\Ccnt{C_{\rm cnt}}
\def\Crate{C_{\rm rate}}

\newcounter{statement}
\newenvironment{statement}[2][!]{%
\vskip3mm
\hrule
\hrule
\hrule
\vskip1mm
\noindent%
\refstepcounter{statement}%
\bf#2~\thestatement%
\ifthenelse{\equal{#1}{!}}{.\ }{~(#1).\ }%
\it%
}{%
\vskip1mm
\hrule
\hrule
\hrule
\vskip2mm
}

\newenvironment{theorem}[1][!]{\begin{statement}[#1]{Theorem}}{\end{statement}}
\newenvironment{lemma}[1][!]{\begin{statement}[#1]{Lemma}}{\end{statement}}

\newenvironment{corollary}[1][!]{\begin{statement}[#1]{Corollary}}{\end{statement}}
\newenvironment{remark}[1][!]{\begin{statement}[#1]{Remark}}{\end{statement}}
\newenvironment{algorithm}[1][!]{\begin{statement}[#1]{Algorithm}}{\end{statement}}

\title{Goal-oriented adaptive finite element methods with optimal computational complexity}
\author{Roland Becker}
\author{Gregor Gantner}
\author{Michael Innerberger}
\author{Dirk Praetorius}

\address{Université de Pau et des Pays de l’Adour, IPRA-LMAP, Avenue de l’Université BP 1155, 64013 PAU Cedex, France}
\email{roland.becker@univ-pau.fr}

\address{Korteweg-de Vries (KdV) Institute for Mathematics, University of Amsterdam, P.O. Box 94248, 1090 GE Amsterdam, The Netherlands.}
\email{g.gantner@uva.nl}

\address{TU Wien, Institute for Analysis and Scientific Computing, Wiedner Hauptstr. 8-10/E101/4, 1040 Vienna, Austria}
\email{michael.innerberger@asc.tuwien.ac.at \quad \rm (corresponding author)}

\address{TU Wien, Institute for Analysis and Scientific Computing, Wiedner Hauptstr. 8-10/E101/4, 1040 Vienna, Austria}
\email{dirk.praetorius@asc.tuwien.ac.at}

\keywords{Adaptivity, Goal-Oriented Algorithm, Finite Element Method, Convergence of
	Adaptive FEM, Optimal Convergence Rates, Optimal Computational Cost}
\subjclass[2010]{65N30, 65N50, 65Y20, 41A25, 65N22}
\thanks{{\bf Acknowledgement.} The authors thankfully acknowledge support by the Austrian Science Fund (FWF) through the doctoral school \emph{Dissipation and dispersion in nonlinear PDEs} (grant W1245), the SFB \emph{Taming complexity in partial differential systems} (grant SFB F65), the stand-alone project \emph{Computational nonlinear PDEs} (grant P33216), and the Erwin Schrödinger Fellowship \emph{Optimal adaptivity for space-time methods} (grant J4379).}

\begin{document}

%%%%%%%%%%%%%%%%%%%%%%%%%%%%%%%%%%%%%%%%%%%%%%%%%%%%%%%%%%%%%%%%%%%%%%%%%%%%%%%%%%%
%%%%%%%%%%%%%%%%%%%%%%%%%%%%%%%%%%%%%%%%%%%%%%%%%%%%%%%%%%%%%%%%%%%%%%%%%%%%%%%%%%%
\begin{abstract}
We consider a linear symmetric and elliptic PDE and a linear goal functional. We design and analyze a goal-oriented adaptive finite element method, which steers the adaptive mesh-refinement as well as the approximate solution of the arising linear systems by means of a contractive iterative solver like the optimally preconditioned conjugate gradient method or geometric multigrid. We prove linear convergence of the proposed adaptive algorithm with optimal algebraic rates. Unlike prior work, we do not only consider rates with respect to the number of degrees of freedom but even prove optimal complexity, i.e., optimal convergence rates with respect to the total computational cost.
\color{black}
\end{abstract}
%%%%%%%%%%%%%%%%%%%%%%%%%%%%%%%%%%%%%%%%%%%%%%%%%%%%%%%%%%%%%%%%%%%%%%%%%%%%%%%%%%%
%%%%%%%%%%%%%%%%%%%%%%%%%%%%%%%%%%%%%%%%%%%%%%%%%%%%%%%%%%%%%%%%%%%%%%%%%%%%%%%%%%%

\maketitle
\thispagestyle{fancy}

%!TEX root = main.tex

%%%%%%%%%%%%%%%%%%%%%%%%%%%%%%%%%%%%%%%%%%%%%%%%%%%%%%%%%%%%%%%%%%%%%%%%%%%%%%%%%%%
%%%%%%%%%%%%%%%%%%%%%%%%%%%%%%%%%%%%%%%%%%%%%%%%%%%%%%%%%%%%%%%%%%%%%%%%%%%%%%%%%%%
\section{Introduction}
%%%%%%%%%%%%%%%%%%%%%%%%%%%%%%%%%%%%%%%%%%%%%%%%%%%%%%%%%%%%%%%%%%%%%%%%%%%%%%%%%%%
%%%%%%%%%%%%%%%%%%%%%%%%%%%%%%%%%%%%%%%%%%%%%%%%%%%%%%%%%%%%%%%%%%%%%%%%%%%%%%%%%%%

Let $\Omega \subset \R^d$ be a bounded Lipschitz domain, $d \ge 2$.
For given $f \in L^2(\Omega)$ and $\f \in [L^2(\Omega)]^d$, we consider a linear symmetric and elliptic partial differential equation
\begin{align}\label{eq:strongform}
\begin{split}
	-\div \A \nabla u^\star + cu^\star &= f + \div \f\hphantom{0} \quad \text{in } \Omega,\\
	u^\star &= 0 \hphantom{f + \div \f} \quad \text{on } \Gamma := \partial\Omega,
\end{split}
\end{align}
where $\A(x) \in \R^{d \times d}_{\mathrm{sym}}$ is symmetric and $c(x) \in \R$.
As usual, we assume that $\A, c \in L^\infty(\Omega)$, that $\A$ is uniformly positive definite and that the weak form (see~\eqref{eq:primal:formulation} below) fits into the setting of the Lax--Milgram lemma.
Standard adaptivity aims to approximate the unknown solution $u^\star \in H^1_0(\Omega)$ of~\eqref{eq:strongform} in the energy norm at optimal rate; see~\cite{doerfler1996,mns2000,bdd2004,stevenson2007,ckns2008,cn2012,ffp2014} for adaptive finite element methods (AFEMs) and~\cite{axioms} for an overview of available results. Instead, the quantity of interest for goal-oriented adaptivity is only some functional value of the unknown solution $u^\star \in H^1_0(\Omega)$ of~\eqref{eq:strongform}, and the present paper aims to compute the linear goal functional
\begin{align}
\label{eq:goal-functional}
	G(u^\star) := \int_\Omega \big( gu^\star - \g \cdot \nabla u^\star \big) \d{x},
\end{align}
for given $g \in L^2(\Omega)$ and $\g \in [L^2(\Omega)]^d$.
To approximate $G(u^\star)$ accurately, it is not necessary (and might even waste computational time) to accurately approximate the solution $u^\star$ on the whole computational domain.
Due to this potential decrease of computational cost, goal-oriented adaptivity is of high relevance in practice as well as in mathematical research; see, e.g., \cite{br01, br03, eehj95, gs2002} for some prominent contributions.

The present work formulates a goal-oriented adaptive finite element method (GOAFEM), where the sought goal $G(u^\star)$ is approximated by some computable $G_\ell$  such that
\begin{align}
	| G(u^\star) - G_\ell | \xrightarrow{\ell \to \infty} 0
	\quad \text{even at optimal algebraic rate}.
\end{align}
The earlier works~\cite{ms2009,bet2011,pointabem,fpz2014} are essentially concerned with optimal convergence rates for GOAFEM, where all arising linear FEM systems are solved exactly. While~\cite{pointabem,fpz2014} particularly aim to transfer ideas from the AFEM analysis of~\cite{ckns2008,axioms} to GOAFEM for general elliptic PDEs, the seminal work~\cite{ms2009} considers the Poisson model problem and additionally addresses the total computational cost by formulating realistic assumptions on a generic inexact solver (called GALSOLVE in~\cite{stevenson2007,ms2009}).

The focus of the present work is also on the iterative (and hence inexact) solution of the arising FEM systems. However, we avoid any \emph{realistic assumptions} on the solver, but rather rely on energy contraction per solver step, which is proved to hold for the preconditioned CG method with optimal multilevel additive Schwarz preconditioner~\cite{cnx12} or the geometric multigrid method~\cite{wz17}. In the proposed GOAFEM algorithm, the termination of such a contractive iterative solver is then based on appropriate computable {\sl a posteriori} error estimates by a similar criterion as in~\cite{stevenson2007,ms2009}. We discuss several implementations of such termination criteria and prove that these allow to control the total computational cost of computing the approximate goal value $G_\ell$, where we already stress now that $G(u^\star) \approx G_\ell = G(u_\ell) + R_\ell$, where $u_\ell \approx u^\star$ is a FEM approximation of $u^\star$ and $R_\ell$ is a residual correction related to inexact solution of the FEM formulation.
While~\cite{ms2009} shows algebraic convergence with optimal rates with respect to the overall computational cost for the final iterates on every level for sufficiently small adaptivity parameters (for mesh-refinement and solver termination), our main contribution is \emph{full linear convergence}, i.e., linear convergence of the estimator product independently of the algorithmic decision for either mesh-refinement or solver step, even for arbitrary adaptivity parameters.
An immediate consequence is that the convergence rate of the computed solutions with respect to the number of elements will be the same as with respect to the overall computational cost (i.e., the cumulative computational time).
Moreover, for sufficiently small adaptivity parameters, we show convergence with optimal rates with respect to the number of elements and, hence, with respect to the overall computational cost.
This extends the results of~\cite{ms2009} to the present setting of symmetric second-order linear elliptic PDEs.
Finally, we stress that, unlike~\cite{ms2009}, our GOAFEM algorithm does not require any inner loop for data approximation and therefore does not require different (but still nested) meshes for the primal and dual problem.
Overall, the present paper thus provides further mathematical understanding for bridging the gap between applied GOAFEM and theoretical optimality results.

\textbf{Outline.}
In Section~\ref{sec:fem}, we present our GOAFEM algorithm (Algorithm~\ref{algorithm}) and the details of its individual steps.
This includes the details of our finite element discretization as well as the precise assumptions for the iterative solver, the marking strategy, and the error estimators.
We then state in Section~\ref{sec:results} that Algorithm~\ref{algorithm} leads to linear convergence for arbitrary stopping parameters (Theorem~\ref{theorem:linear-convergence}) and even achieves optimal rates with respect to the total computational cost if the adaptivity parameters are sufficiently small (Theorem~\ref{theorem:optimal-convergence}).
We emphasize that linear convergence applies to all steps of the adaptive strategy, independently of whether the algorithm decides for one solver step or one step of local mesh-refinement.
This turns out to be the key argument for optimal rates with respect to the total computational cost (see Corollary~\ref{corollary:optimal-convergence}).
Section~\ref{section:termination} comments on alternative termination criteria for the iterative solver.
Section~\ref{sec:numerics} then illustrates our theoretical findings with numerical experiments.
Finally, we give a proof of our main Theorems~\ref{theorem:linear-convergence} and~\ref{theorem:optimal-convergence} in Section~\ref{sec:proof1} and Section~\ref{sec:proof2}, respectively.

\textbf{Notation.}
In the following text, we write $a \lesssim b$ for $a, b \in  \R$ if there exists a constant $C > 0$ (which is independent of the mesh width $h$) such that $a \leq C \, b$.
If there holds $a \lesssim b \lesssim a$, we abbreviate this by $a \simeq b$.
Furthermore, we denote by $\#A$ the cardinality of a finite set $A$ and by $|\omega|$ the $d$-dimensional Lebesgue measure of a subset $\omega \subset \R^d$.
%\clearpage
%!TEX root = main.tex

%%%%%%%%%%%%%%%%%%%%%%%%%%%%%%%%%%%%%%%%%%%%%%%%%%%%%%%%%%%%%%%%%%%%%%%%%%%%%%%%%%%
%%%%%%%%%%%%%%%%%%%%%%%%%%%%%%%%%%%%%%%%%%%%%%%%%%%%%%%%%%%%%%%%%%%%%%%%%%%%%%%%%%%
\section{Goal-oriented adaptive finite element method}\label{sec:fem}
%%%%%%%%%%%%%%%%%%%%%%%%%%%%%%%%%%%%%%%%%%%%%%%%%%%%%%%%%%%%%%%%%%%%%%%%%%%%%%%%%%%
%%%%%%%%%%%%%%%%%%%%%%%%%%%%%%%%%%%%%%%%%%%%%%%%%%%%%%%%%%%%%%%%%%%%%%%%%%%%%%%%%%%

%%%%%%%%%%%%%%%%%%%%%%%%%%%%%%%%%%%%%%%%%%%%%%%%%%%%%%%%%%%%%%%%%%%%%%%%%%%%%%%%%%%
\subsection{Variational formulation}
%%%%%%%%%%%%%%%%%%%%%%%%%%%%%%%%%%%%%%%%%%%%%%%%%%%%%%%%%%%%%%%%%%%%%%%%%%%%%%%%%%%

Defining the symmetric bilinear form 
\begin{align}
	a(u, v) := \int_\Omega \A \nabla u \cdot \nabla v \d{x} + \int_\Omega cuv \d{x},
\end{align}
we suppose that $a(\cdot,\cdot)$ is continuous and elliptic on $H^1_0(\Omega)$ and thus fits into the setting of the Lax--Milgram lemma, i.e., there exist constants $0 < \Cell \leq \Ccnt < \infty$ such that
\begin{equation*}
	\Cell \norm{u}{H^1_0(\Omega)}^2
	\leq
	a(u, u)
	\quad\text{and}\quad
	a(u,v) \leq \Ccnt \norm{u}{H^1_0(\Omega)} \norm{v}{H^1_0(\Omega)}
	\quad\text{for all }
	u,v \in H^1_0(\Omega).
\end{equation*}
In particular, $a(\cdot, \cdot)$ is a scalar product that yields an equivalent norm $\enorm{v}^2 := a(v,v)$ on $H^1_0(\Omega)$.
The weak formulation of~\eqref{eq:strongform} reads
\begin{align}\label{eq:primal:formulation}
	a(u^\star, v) = F(v) := \int_\Omega \big( fv \d{x} - \f \cdot \nabla v \big) \d{x}
	\quad \text{for all } v \in H^1_0(\Omega).
\end{align}
The Lax--Milgram lemma proves existence and uniqueness of the solution $u^\star \in H^1_0(\Omega)$ of~\eqref{eq:primal:formulation}.
The same argument applies and proves that the dual problem 
\begin{align}\label{eq:dual:formulation}
	a(v, z^\star) =  G(v)
	\quad \text{for all } v \in H^1_0(\Omega)
\end{align}
admits a unique solution $z^\star \in H^1_0(\Omega)$, where the linear goal functional $G \in H^{-1}(\Omega) := H^1_0(\Omega)'$ is defined by~\eqref{eq:goal-functional}.

\begin{remark}
\label{rem:neumann-boundary}
	For ease of presentation, we restrict our model problem~\eqref{eq:strongform} to homogeneous Dirichlet boundary conditions.
	We note, however, that for mixed homogeneous Dirichlet and inhomogeneous Neumann boundary conditions our main results hold true with the obvious modifications.
	In particular, with the partition $\partial \Omega = \overline{\Gamma}_D \cup \overline{\Gamma}_N$ into Dirichlet boundary $\Gamma_D$ with $|\Gamma_D| > 0$ and Neumann boundary $\Gamma_N$, the space $H^1_0(\Omega)$ (and its discretization) has to be replaced by $H^1_D(\Omega) := \set{v \in H^1(\Omega)}{v|_{\Gamma_D} = 0 \text{ in the sense of traces}}$ and the Neumann data has to be given in $L^2(\Gamma_N)$.
	Furthermore, the coefficient $\f$ must vanish in a neighborhood of $\Gamma_N$ to go from the strong form~\eqref{eq:strongform} to the weak form~\eqref{eq:primal:formulation} via integration by parts.
\end{remark}

%%%%%%%%%%%%%%%%%%%%%%%%%%%%%%%%%%%%%%%%%%%%%%%%%%%%%%%%%%%%%%%%%%%%%%%%%%%%%%%%%%%
\subsection{Finite element discretization and solution}
%%%%%%%%%%%%%%%%%%%%%%%%%%%%%%%%%%%%%%%%%%%%%%%%%%%%%%%%%%%%%%%%%%%%%%%%%%%%%%%%%%%

For a conforming triangulation  $\TT_\coarse$ of $\Omega$ into compact simplices and a polynomial degree $p \ge 1$, let
\begin{align}
	\XX_\coarse := \set{v_\coarse \in H^1_0(\Omega)}{\forall T \in \TT_\coarse \quad v_\coarse|_T \text{ is a polynomial of degree } \le p}.
\end{align}
To obtain conforming finite element approximations $u^\star \approx u_\coarse \in \XX_\coarse$ and $z^\star \approx z_\coarse \in \XX_\coarse$, we consider the Galerkin discretizations of~\eqref{eq:primal:formulation}--\eqref{eq:dual:formulation}.
First, we note that the Lax--Milgram lemma yields the existence and uniqueness of \emph{exact} discrete solutions $u_\coarse^\star, z_\coarse^\star \in \XX_\coarse$, i.e., there holds that
\begin{align}
\label{eq:discrete:formulation}
	a(u_\coarse^\star, v_\coarse) = F(v_\coarse)
	\quad \text{and} \quad
	a(v_\coarse, z_\coarse^\star) =  G(v_\coarse)
	\,\,  \text{ for all } v_\coarse \in \XX_\coarse.
\end{align}
In practice, the discrete systems~\eqref{eq:discrete:formulation} are rarely solved exactly (or up to machine precision).
Instead, a suitable iterative solver is employed, which yields \emph{approximate} discrete solutions $u_{\coarse}^{m}, z_{\coarse}^{n} \in \XX_\coarse$.
We suppose that this iterative solver is contractive, i.e., for all $m, n \in \N$, it holds that
\begin{align}\label{eq:solver}
	\enorm{u_\coarse^\star - u_{\coarse}^{m}}
	\leq
	\qctr \, \enorm{u_\coarse^\star - u_{\coarse}^{m-1}}
	\,\, \text{ and } \,\, 
	\enorm{z_\coarse^\star - z_{\coarse}^{n}}
	\leq
	\qctr \, \enorm{z_\coarse^\star - z_{\coarse}^{n-1}},
\end{align}
where $0 < \qctr < 1$ is a generic constant and, in particular, independent of $\XX_\coarse$.
Assumption~\eqref{eq:solver} is satisfied, e.g., for an optimally preconditioned conjugate gradient (PCG) method (see~\cite{cnx12}) or geometric multigrid solvers (see~\cite{wz17}); see also the discussion in~\cite{banach2}.
We note that these solvers are also guaranteed to satisfy the \emph{realistic assumptions} from~\cite{stevenson2007,ms2009} (which require that any initial energy error can be improved by a factor $0<\tau<1$ at $\OO(|\log(\tau)|\#\TT_\coarse)$ cost).
However, while~\eqref{eq:solver} is slightly less general, it allows to prove \emph{full linear convergence}; see Theorem~\ref{theorem:linear-convergence} below.

%%%%%%%%%%%%%%%%%%%%%%%%%%%%%%%%%%%%%%%%%%%%%%%%%%%%%%%%%%%%%%%%%%%%%%%%%%%%%%%%%%%
\subsection{Discrete goal quantity}
%%%%%%%%%%%%%%%%%%%%%%%%%%%%%%%%%%%%%%%%%%%%%%%%%%%%%%%%%%%%%%%%%%%%%%%%%%%%%%%%%%%

To approximate $G(u^\star)$, we proceed as in~\cite{gs2002}:
For any $u_\coarse, z_\coarse \in \XX_\coarse$, it holds that
\begin{align*}
	G(u^\star) - G(u_\coarse) 
	= G(u^\star - u_\coarse)
	\stackrel{\mathclap{\eqref{eq:dual:formulation}}}{=}
	a(u^\star &- u_\coarse , z^\star)
	= a(u^\star - u_\coarse , z^\star - z_\coarse) + a(u^\star - u_\coarse , z_\coarse)
	\\&
	\stackrel{\mathclap{\eqref{eq:primal:formulation}}}{=}
	a(u^\star - u_\coarse , z^\star - z_\coarse) + \big[ F(z_\coarse) - a(u_\coarse , z_\coarse) \big].
\end{align*}
Defining the \emph{discrete quantity of interest}
\begin{align}
\label{eq:def:goal}
	G_\coarse(u_\coarse, z_\coarse) := G(u_\coarse) + \big[ F(z_\coarse) - a(u_\coarse , z_\coarse) \big],
\end{align}
the goal error can be controlled by means of the Cauchy--Schwarz inequality
\begin{equation}
\label{eq:goal-norm-estimate}
	\big| G(u^\star) - G_\coarse(u_\coarse, z_\coarse) \big|
	\le \big| a(u^\star - u_\coarse , z^\star - z_\coarse) \big| 
	\le \enorm{u^\star - u_\coarse} \, \enorm{z^\star - z_\coarse}.
\end{equation}
We note that the additional term in~\eqref{eq:def:goal} is the residual of the discrete primal problem~\eqref{eq:discrete:formulation} evaluated at an arbitrary function $z_\coarse \in \XX_\coarse$ and hence $G(u_\coarse^\star) = G_\coarse(u_\coarse^\star, z_\coarse)$.

In the following, we design an adaptive algorithm that provides a computable upper bound to~\eqref{eq:goal-norm-estimate} which tends to zero at optimal algebraic rate with respect to the number of elements $\#\TT_\coarse$ as well as with respect to the total computational cost.

%%%%%%%%%%%%%%%%%%%%%%%%%%%%%%%%%%%%%%%%%%%%%%%%%%%%%%%%%%%%%%%%%%%%%%%%%%%%%%%%%%%
\subsection{Mesh refinement}
%%%%%%%%%%%%%%%%%%%%%%%%%%%%%%%%%%%%%%%%%%%%%%%%%%%%%%%%%%%%%%%%%%%%%%%%%%%%%%%%%%%

Let $\TT_0$ be a given conforming triangulation of $\Omega$.
We suppose that the mesh-refinement is a deterministic and fixed strategy, e.g., newest vertex bisection~\cite{stevenson2008}. 
For each conforming triangulation $\TT_\coarse$ and marked elements $\MM_\coarse \subseteq \TT_\coarse$, let $\TT_\fine := \refine(\TT_\coarse,\MM_\coarse)$ be the coarsest conforming triangulation, where all $T \in \MM_\coarse$ have been refined, i.e., $\MM_\coarse \subseteq \TT_\coarse \backslash \TT_\fine$. 
We write $\TT_\fine \in \T(\TT_\coarse)$, if $\TT_\fine$ results from $\TT_\coarse$ by finitely many steps of refinement.
To abbreviate notation, let $\T:=\T(\TT_0)$.
We note that the order on $\T$ is respected by the finite element spaces, i.e., $\TT_\fine \in \T(\TT_\coarse)$ implies that $\XX_\coarse \subseteq \XX_\fine$.

We further suppose that each refined element has at least two sons, i.e.,
\begin{align}\label{eq:mesh-sons}
	\#(\TT_\coarse \backslash \TT_\fine) + \#\TT_\coarse \leq \#\TT_\fine
	\quad \text{for all }
	\TT_\coarse \in \T \text{ and all } \TT_\fine \in \T(\TT_\coarse),
\end{align}
and that the refinement rule satisfies the mesh-closure estimate
\begin{align}\label{eq:mesh-closure}
	\#\TT_\ell - \#\TT_0 \leq \Ccls\,\sum_{j=0}^{\ell-1}\#\MM_j
	\quad \text{for all }
	\ell \in \N,
\end{align}
where $\Ccls>0$ depends only on $\TT_0$.
For newest vertex bisection, this has been proved under an additional admissibility assumption on $\TT_0$ in~\cite{bdd2004, stevenson2008} and for 2D even without any additional assumption in \cite{kpp2013}.
Finally, we suppose that the overlay estimate holds, i.e., for all triangulations $\TT_\coarse,\TT_\fine \in \T$, there exists a common refinement $\TT_\coarse \oplus \TT_\fine \in \T(\TT_\coarse) \cap \T(\TT_\fine)$ which satisfies that
\begin{align}\label{eq:mesh-overlay}
\#(\TT_\coarse \oplus \TT_\fine) \leq \#\TT_\coarse + \#\TT_\fine - \#\TT_0,
\end{align}
which has been proved in~\cite{stevenson2007,ckns2008} for newest vertex bisection.

%%%%%%%%%%%%%%%%%%%%%%%%%%%%%%%%%%%%%%%%%%%%%%%%%%%%%%%%%%%%%%%%%%%%%%%%%%%%%%%%%%%
\subsection{Estimator properties}
%%%%%%%%%%%%%%%%%%%%%%%%%%%%%%%%%%%%%%%%%%%%%%%%%%%%%%%%%%%%%%%%%%%%%%%%%%%%%%%%%%%

For $\TT_\coarse \in \T$ and $v_\coarse \in \XX_\coarse$, let
\begin{equation*}
	\eta_\coarse(T, v_\coarse) \geq 0
	\quad\text{and}\quad
	\zeta_\coarse(T, v_\coarse) \geq 0
	\quad\text{for all }
	T \in \TT_\coarse
\end{equation*}
be given refinement indicators.
For $\mu_\coarse \in \{\eta_\coarse, \zeta_\coarse\}$, we use the usual convention that
\begin{align}
	\mu_\coarse(v_\coarse) := \mu_\coarse(\TT_\coarse, v_\coarse),
	\quad \text{where} \quad
	\mu_\coarse(\UU_\coarse, v_\coarse) = \bigg( \sum_{T \in \UU_\coarse} \mu_\coarse(T, v_\coarse)^2 \bigg)^{1/2}
\end{align}
for all $v_\coarse \in \XX_\coarse$ and all $\UU_\coarse \subseteq \TT_\coarse$.

We suppose that the estimators $\eta_\coarse$ and $\zeta_\coarse$ satisfy the so-called \emph{axioms of adaptivity} (which are designed for, but not restricted to, weighted-residual error estimators) from~\cite{axioms}:
There exist constants $\Cstab, \Crel, \Cdrel > 0$ and $0 < \qred < 1$ such that for all $\TT_\coarse \in \T(\TT_0)$ and all $\TT_\fine \in \T(\TT_\coarse)$, the following assumptions are satisfied:
\renewcommand{\theenumi}{{A\arabic{enumi}}}
\begin{enumerate}
	\bf
	\item\label{assumption:stab} Stability:
	\rm 
	For all $v_\fine \in \XX_\fine$, $v_\coarse \in \XX_\coarse$, and $\UU_\coarse \subseteq \TT_\fine \cap \TT_\coarse$, it holds that
	\begin{align*}
		\big| \eta_\fine(\UU_\coarse, v_\fine) - \eta_\coarse(\UU_\coarse, v_\coarse) \big|
		+ \big| \zeta_\fine(\UU_\coarse, v_\fine) - \zeta_\coarse(\UU_\coarse, v_\coarse) \big|
		&\le
		\Cstab \, \enorm{v_\fine - v_\coarse}.
	\end{align*}
	\bf
	\item\label{assumption:red} Reduction:
	\rm
	For all $v_\coarse \in \XX_\coarse$, it holds that
	\begin{align*}
		\eta_\fine(\TT_\fine \backslash \TT_\coarse, v_\coarse) 
		&\le
		\qred \, \eta_\coarse(\TT_\coarse \backslash \TT_\fine, v_\coarse)
		\quad \text{ and } \quad
		\zeta_\fine(\TT_\fine \backslash \TT_\coarse, v_\coarse) 
		\le
		\qred \, \zeta_\coarse(\TT_\coarse \backslash \TT_\fine, v_\coarse).
	\end{align*}
	\bf
	\item\label{assumption:rel} Reliability:
	\rm
	The Galerkin solutions $u_\coarse^\star, z_\coarse^\star \in \XX_\coarse$ to~\eqref{eq:discrete:formulation} satisfy that
	\begin{align*}
		\enorm{u^\star - u_\coarse^\star} 
		&\le
		\Crel \, \eta_\coarse(u_\coarse^\star)
		\quad \text{ and } \quad
		\enorm{z^\star - z_\coarse^\star} 
		\le
		\Crel \, \zeta_\coarse(z_\coarse^\star).
	\end{align*} 
	\bf
	\item\label{assumption:drel} Discrete reliability:
	\rm
	The Galerkin solutions $u_\coarse^\star, z_\coarse^\star \in \XX_\coarse$ and $u_\fine^\star, z_\fine^\star \in \XX_\fine$ to~\eqref{eq:discrete:formulation} satisfy that
	\begin{align*}
		\enorm{u_\fine^\star - u_\coarse^\star} 
		&\le
		\Cdrel \, \eta_\coarse(\TT_\coarse \backslash \TT_\fine, u_\coarse^\star)
		\quad \text{ and } \quad
		\enorm{z_\fine^\star - z_\coarse^\star} 
		\le
		\Cdrel \, \zeta_\coarse(\TT_\coarse \backslash \TT_\fine, z_\coarse^\star). 
	\end{align*} 
\end{enumerate}

By assumptions~\eqref{assumption:stab} and~\eqref{assumption:rel}, we can estimate for every discrete function $w_\coarse \in \XX_\coarse$ the errors in the energy norm of the primal and the dual problem by
\begin{equation*}
	\enorm{u^\star - w_\coarse}
	\leq
	C \, \big[ \eta_\coarse(w_\coarse) + \enorm{u_\coarse^\star - w_\coarse} \big]
	\quad \text{and} \quad
	\enorm{z^\star - w_\coarse}
	\leq
	C \, \big[ \zeta_\coarse(w_\coarse) + \enorm{z_\coarse^\star - w_\coarse} \big],
\end{equation*}
respectively, where $C = \max\{ \Crel, \Crel\Cstab + 1 \} > 0$.
Together with~\eqref{eq:goal-norm-estimate}, we then obtain that the goal error for approximations $u_\coarse^m \approx u_\coarse^\star$ and $z_\coarse^n \approx z_\coarse^\star$ in $\XX_\coarse$ is bounded by
\begin{align}
\label{eq:goal:aposteriori}
	\big| G(u^\star) - G_\coarse(u_{\coarse}^{m}, z_{\coarse}^{n}) \big|
	\leq C^2 \,
	\big[ \, \eta_\coarse(u_{\coarse}^{m}) + \enorm{u_\coarse^\star - u_{\coarse}^{m}} \, \big] \,
	\big[ \, \zeta_\coarse(z_{\coarse}^{n}) + \enorm{z_\coarse^\star - z_{\coarse}^{n}} \, \big].
\end{align}
In the following sections, we provide building blocks for our adaptive algorithm that allow to control the arising estimators (by a suitable marking strategy) as well as the arising norms in the upper bound of~\eqref{eq:goal:aposteriori} (by an appropriate stopping criterion for the iterative solver).

%%%%%%%%%%%%%%%%%%%%%%%%%%%%%%%%%%%%%%%%%%%%%%%%%%%%%%%%%%%%%%%%%%%%%%%%%%%%%%%%%%%
\subsection{Marking strategy}
%%%%%%%%%%%%%%%%%%%%%%%%%%%%%%%%%%%%%%%%%%%%%%%%%%%%%%%%%%%%%%%%%%%%%%%%%%%%%%%%%%%

We suppose that the refinement indicators $\eta_\coarse(T, u_\coarse^m)$ and $\zeta_\coarse(T, z_\coarse^n)$ for some $m,n \in \N$ are used to mark a subset $\MM_\coarse \subseteq \TT_\coarse$ of elements for refinement, which, for fixed marking parameter $0 < \theta \leq 1$, satisfies that
\begin{equation}
\label{eq:abstract-marking}
    2 \theta \eta_\coarse(u_\coarse^m)^2 \zeta_\coarse(z_\coarse^n)^2
    \le
	\eta_\coarse(\MM_\coarse, u_\coarse^m)^2 \zeta_\coarse(z_\coarse^n)^2
	+ \zeta_\coarse(\MM_\coarse, z_\coarse^n)^2 \eta_\coarse(u_\coarse^m)^2.
\end{equation}

\begin{remark}\label{rem:marking}
	Given $0 < \vartheta \leq 1$, possible choices of marking strategies satisfying assumption~\eqref{eq:abstract-marking} are the following:
	\begin{enumerate}[label={\rm(\alph*)}]
		\item The strategy proposed in~\cite{bet2011} defines the weighted estimator
		\begin{equation*}
			\rho_\coarse(T, u_\coarse^m, z_\coarse^n)^2
			:=
			\eta_\coarse(T, u_\coarse^m)^2 \zeta_\coarse(z_\coarse^n)^2
			+ \eta_\coarse(u_\coarse^m)^2 \zeta_\coarse(T, z_\coarse^n)^2
		\end{equation*}
		and then determines a set $\MM_\coarse \subseteq \TT_\coarse$ such that
		\begin{equation}
		\label{eq:bet-marking}
			\vartheta \, \rho_\coarse(u_\coarse^m, z_\coarse^n)
			\le
			\rho_\coarse(\MM_\coarse, u_{\coarse}^{m}, z_\coarse^n)
		\end{equation}
		which is the D\"orfler marking criterion introduced in~\cite{doerfler1996} and well-known in the context of AFEM analysis; see, e.g.,~\cite{axioms}.
		This strategy satisfies~\eqref{eq:abstract-marking} with $\theta = \vartheta^2$.
		
		\item The strategy proposed in~\cite{ms2009} determines sets $\overline\MM_\coarse^{u}, \overline\MM_\coarse^{z} \subseteq \TT_\coarse$ such that
		\begin{equation}
		\label{eq:doerfler}
			\vartheta \, \eta_\coarse(u_{\coarse}^{m})
			\le
			\eta_\ell(\overline\MM_\coarse^{u},u_{\coarse}^{m})
			\quad \text{and} \quad
			\vartheta \, \zeta_\coarse(z_{\coarse}^{n})
			\le
			\zeta_\coarse(\overline\MM_\coarse^{z}, z_{\coarse}^{n})
		\end{equation}
		and then chooses $\MM_\coarse := \arg\min \{ \, \#\overline\MM_\coarse^{u} \,,\, \#\overline\MM_\coarse^{z} \, \}$.
		This strategy satisfies~\eqref{eq:abstract-marking} with $\theta = \vartheta^2/2$.
		
		\item A more aggressive variant of~{\rm(b)} was proposed in~\cite{fpz2014}:
		Let $\overline\MM_\coarse^{u}$ and $\overline\MM_\coarse^{z}$ as above.
		Then, choose $\MM_\coarse^{u} \subseteq \overline\MM_\coarse^{u}$ and $\MM_\coarse^{z} \subseteq \overline\MM_\coarse^{z}$ with $\#\MM_\coarse^{u} =  \#\MM_\coarse^{z} =  \min\{ \, \#\overline\MM_\coarse^{u} \,,\, \#\overline\MM_\coarse^{z} \, \}$.
		Finally, define $\MM_\coarse := \MM_\coarse^{u} \cup \MM_\coarse^{z}$.
		Again, this strategy satisfies~\eqref{eq:abstract-marking} with $\theta = \vartheta^2/2$.
	\end{enumerate}
	Note that our main results of Theorem~\ref{theorem:linear-convergence} and~\ref{theorem:optimal-convergence} below hold true for all presented marking criteria {\rm(a)--(c)}.
	For our numerical experiments, we focus on criterion {\rm(a)}, which empirically tends to achieve slightly better performance in practice.
\end{remark}

%%%%%%%%%%%%%%%%%%%%%%%%%%%%%%%%%%%%%%%%%%%%%%%%%%%%%%%%%%%%%%%%%%%%%%%%%%%%%%%%%%%
\subsection{Adaptive algorithm}\label{subsec:algorithm}
%%%%%%%%%%%%%%%%%%%%%%%%%%%%%%%%%%%%%%%%%%%%%%%%%%%%%%%%%%%%%%%%%%%%%%%%%%%%%%%%%%%

Any adaptive algorithm strives to drive down the bound in~\eqref{eq:goal:aposteriori}.
However, the errors of the iterative solver, $\enorm{u_\coarse^\star - u_{\coarse}^{m}}$ and $\enorm{z_\coarse^\star - z_{\coarse}^{n}}$, cannot be computed in general since the exact discrete solutions $u_\coarse^\star, z_\coarse^\star \in \XX_\coarse$ to~\eqref{eq:discrete:formulation} are unknown and will not be computed.
Thus, we note that~\eqref{eq:solver} and the triangle inequality prove that
\begin{subequations}
	\label{eq:contraction}
	\begin{align}
		(1 - \qctr) \, \enorm{u_\coarse^\star - u_{\coarse}^{m-1}}
		\le 
		\enorm{u_{\coarse}^{m} - u_{\coarse}^{m-1}}
		\le
		(1 + \qctr) \, \enorm{u_\coarse^\star - u_{\coarse}^{m-1}}
	\end{align}
	as well as 
	\begin{align}
		(1 - \qctr) \, \enorm{z_\coarse^\star - z_{\coarse}^{n-1}}
		\le
		\enorm{z_{\coarse}^{n} - z_{\coarse}^{n-1}}
		\le
		(1 + \qctr) \, \enorm{z_\coarse^\star - z_{\coarse}^{n-1}}.
	\end{align}
\end{subequations}
With $\Cgoal = \max\{ \Crel, \Crel\Cstab + 1 \} \, \big(1 + \qctr / (1-\qctr)\big)$, \eqref{eq:goal:aposteriori} leads to
\begin{align}
\label{eq2:goal:aposteriori}
	\hspace*{-2mm}
	\big| G(u^\star) - G_\coarse(u_{\coarse}^{m}, z_{\coarse}^{n}) \big|
	\leq
	\Cgoal^2 \, \big[ \eta_\coarse(u_{\coarse}^{m}) + \enorm{u^m_\coarse - u^{m-1}_\coarse} \big]  \big[ \zeta_\coarse(z_{\coarse}^{n}) + \enorm{z^n_\coarse - z^{n-1}_\coarse} \big],
\end{align}
which is a computable upper bound to the goal error if $m, n \geq 1$.
Moreover, given some $\lctr > 0$, this motivates to stop the iterative solvers as soon as
\begin{align*}
	\enorm{u_{\coarse}^{m} - u_{\coarse}^{m-1}} \le \lctr \, \eta_\coarse(u_{\coarse}^{m})
	\quad \text{ and } \quad
	\enorm{z_{\coarse}^{n} - z_{\coarse}^{n-1}} \le \lctr \, \zeta_\coarse(z_{\coarse}^{n})
\end{align*}
to equibalance the contributions of the upper bound in~\eqref{eq2:goal:aposteriori}, which is very similar to the stopping criteria considered in the seminal works~\cite{stevenson2007,ms2009}; further alternative stopping criteria are introduced and analyzed below.
Overall, we thus consider the following adaptive algorithm.

\begin{algorithm}\label{algorithm}
	Let $u_{0}^{0}, z_{0}^{0} \in \XX_0$ be initial guesses. Let $0 < \theta \le 1$ as well as $\lctr > 0$ be arbitrary but fixed marking parameters.
	For all $\ell = 0,1,2,\dots$, perform the following steps~{\rm(i)--(vi)}:
	\begin{itemize}
		\item[\rm(i)] Employ (at least one step of) the iterative solver to compute iterates $u_{\ell}^{1}, \dots, u_{\ell}^{m}$ and $z_{\ell}^{1}, \dots, z_{\ell}^{n}$ together with the corresponding refinement indicators $\eta_\ell(T,u_{\ell}^{k})$ and $\zeta_\ell(T, z_{\ell}^{k})$ for all $T \in \TT_\ell$, until 
		\begin{align}\label{eq:stopping}
			\enorm{u_{\ell}^{m} - u_{\ell}^{ m-1}} \le \lctr \, \eta_\ell(u_{\ell}^{m})
			\quad \text{and} \quad
			\enorm{z_{\ell}^{n} - z_{\ell}^{ n-1}} \le \lctr \, \zeta_\ell(z_{\ell}^{n}).
		\end{align}
		
		\item[\rm(ii)] Define $\m(\ell) := m$ and $\n(\ell) := n$.
		
		\item[\rm(iii)] If $\eta_\ell(u_\ell^m) = 0$ or $\zeta_\ell(z_\ell^m) = 0$, then define $\underline{\ell} := \ell$ and terminate.
		
		\item[\rm(iv)] Otherwise, find a set $\MM_\ell \subseteq \TT_\ell$ such that the marking criterion~\eqref{eq:abstract-marking} is satisfied.
		
		\item[\rm(v)] Generate $\TT_{\ell+1} := \refine(\TT_\ell, \MM_\ell)$.
		
		\item[\rm(vi)] Define the initial guesses $u_{\ell+1}^{ 0} := u_{\ell}^{m}$ and $z_{\ell+1}^{ 0} := z_{\ell}^{n}$ for the iterative solver.
	\end{itemize}
\end{algorithm}

\begin{remark}
\label{rem:marking-parameters}
	Theorem~\ref{theorem:linear-convergence} below proves (linear) convergence for any choice of the marking parameters $0 < \theta \leq 1$ and $\lctr > 0$, and for any of the marking strategies from Remark~\ref{rem:marking}.
	Theorem~\ref{theorem:optimal-convergence} below proves optimal convergence rates (with respect to the number of elements and the total computational cost) if both parameters are sufficiently small (see~\eqref{eq:small-theta} for the precise condition) and if the set $\MM_\ell$ is constructed by one of the strategies from Remark~\ref{rem:marking}, where the respective sets have quasi-minimal cardinality.
\end{remark}

\begin{remark}
\label{rem:estimation}
	Note that Algorithm~\ref{algorithm}{\rm (i)} requires to evaluate the error estimator after each solver step.
	Clearly, it would be favorable to replace $\eta_\ell(u_\ell^m)$ (resp.\ $\zeta_\ell(z_\ell^n)$) by $\eta_\ell(u_\ell^0)$ (resp.\ $\zeta_\ell(z_\ell^0)$) in~\eqref{eq:stopping}.
	Arguing as in~\cite[Lemma~8]{fp18}, this allows to prove convergence of the adaptive strategy, but full linear convergence (Theorem~\ref{theorem:linear-convergence} below) and optimal convergence rates (Theorem~\ref{theorem:optimal-convergence} below) are exptected to fail.
\end{remark}

For each adaptive level $\ell$, Algorithm~\ref{algorithm} performs at least one solver step to compute $u_{\ell}^{m}$ as well as one solver step to compute $z_{\ell}^{n}$.
By definition, $\m(\ell) \ge 1$ is the solver step, for which the discrete solution $u_{\ell}^{ \m(\ell)}$ is accepted (to contribute to the set of marked elements $\MM_\ell$).
Analogously, $\n(\ell) \ge 1$ is  the solver step, for which the discrete solution $z_{\ell}^{ \n(\ell)}$ is accepted (to contribute to $\MM_\ell$).
If the iterative solver for either the primal or the dual problem fails to terminate for some level $\ell \in \N_0$, i.e., \eqref{eq:stopping}~cannot be achieved for finite $m$, or $n$, we define $\m(\ell) := \infty$, or $\n(\ell) := \infty$, respectively, and $\underline{\ell} := \ell$.
With $\k(\ell) := \max\{\m(\ell), \n(\ell)\}$, we define
\begin{align}\label{eq2:stopping}
\begin{split}
	u_{\ell}^{k} &:= u_{\ell}^{\m(\ell)} \hphantom{z_{\ell}^{\n}} \text{for all } k \in \N \text{ with } \m(\ell) < k \le \k(\ell),
	\\
	z_{\ell}^{k} &:= z_{\ell}^{\n(\ell)} \hphantom{u_{\ell}^{\m}} \text{for all } k \in \N \text{ with } \n(\ell) < k \le \k(\ell).
\end{split}
\end{align}
For ease of presentation, we omit the $\ell$-dependence of the indices for final iterates $\m(\ell)$, $\n(\ell)$, and $\k(\ell)$ in the following, if they appear as upper indices and write, e.g., $u_{\ell}^{\m} := u_{\ell}^{\m(\ell)}$ and $u_{\ell}^{\m-1} := u_{\ell}^{\m(\ell)-1}$.
If Algorithm~\ref{algorithm} does not terminate in step~{\rm(iii)} for some $\ell \in \N$, then we define $\underline{\ell} := \infty$.
To formulate the convergence of Algorithm~\ref{algorithm}, we define the ordered set
\begin{align}
\QQ := \set{(\ell,k) \in \N_0^2}{\ell \leq \underline{\ell} \text{ and } 1 \le k \le \k(\ell)},
\quad \text{where} \quad
|(\ell,k)| := k + \sum_{j=0}^{\ell-1} \k(j).
\end{align}
Note that $|(\ell,k)|$ is proportional to the overall number of solver steps to compute the estimator product $\eta_\ell(u_\ell^k)\zeta_\ell(z_\ell^k)$.
Additionally, we sometimes require the notation
\begin{align}
    \QQ_0 := \set{(\ell,k)\in\N_0^2}{\ell\le\underline{\ell} \text{ and } 0\le k \le \k(\ell)}
    = \QQ \cup \set{(\ell,0)\in\N_0^2}{\ell\le \underline\ell}.
\end{align}
To estimate the work necessary to compute a pair $(u_{\ell}^{k}, z_{\ell}^{k}) \in \XX_\ell \times \XX_\ell$, we make the following assumptions which are usually satisfied in practice:
\begin{itemize}
	\item The iterates $u_{\ell}^{k}$ and $z_{\ell}^{k}$ are computed in parallel and each step of the solver in  Algorithm~\ref{algorithm}(i) can be done in linear complexity $\OO(\#\TT_\ell)$;
	
	\item Computation of all indicators $\eta_\ell(T,u_{\ell}^{k})$ and $\zeta_\ell(T, z_{\ell}^{k})$ for $T \in \TT_\ell$ requires $\OO(\#\TT_\ell)$ steps;
	
	\item The marking in Algorithm~\ref{algorithm}{\rm(iv)} can be performed at linear cost $\OO(\#\TT_\ell)$ (according to~\cite{stevenson2007} this can be done for the strategies outlined in Remark~\ref{rem:marking} with $\MM_\ell$ having almost minimal cardinality; moreover, we refer to a recent own algorithm in~\cite{pp2019} with linear cost even for $\MM_\ell$ having minimal cardinality);
	
	\item We have linear cost $\OO(\#\TT_\ell)$ to generate the new mesh $\TT_{\ell+1}$.
\end{itemize}
Since a step $(\ell,k) \in \QQ$ of Algorithm~\ref{algorithm} depends on the full history of preceding steps, the total work spent to compute $(u_{\ell}^{k}, z_{\ell}^{k}) \in \XX_\ell \times \XX_\ell$ is then of order
\begin{equation}
\label{eq:work}
	\work(\ell,k) := \sum_{\substack{ (\ell',k') \in \QQ \\ |(\ell',k')| \le |(\ell,k)| }} \#\TT_{\ell'}
	\quad \text{for all }
	(\ell,k) \in \QQ.
\end{equation}

Finally, we note that Algorithm~\ref{algorithm}{\rm (vi)} employs \emph{nested iteration} to obtain the initial guesses $u^0_{\ell+1}, z^0_{\ell+1}$ of the solver from the final iterates $u^{\m}_{\ell}, z^{\n}_{\ell}$ for the mesh $\TT_\ell$.
According to~\eqref{eq2:goal:aposteriori}, this allows for \textsl{a posteriori} error control for all indices $(\ell, k) \in \QQ_0 \setminus \{(0,0)\}$ beyond the initial step.

%\clearpage
%!TEX root = main.tex

%%%%%%%%%%%%%%%%%%%%%%%%%%%%%%%%%%%%%%%%%%%%%%%%%%%%%%%%%%%%%%%%%%%%%%%%%%%%%%%%%%%
%%%%%%%%%%%%%%%%%%%%%%%%%%%%%%%%%%%%%%%%%%%%%%%%%%%%%%%%%%%%%%%%%%%%%%%%%%%%%%%%%%%
\section{Main results}\label{sec:results}
%%%%%%%%%%%%%%%%%%%%%%%%%%%%%%%%%%%%%%%%%%%%%%%%%%%%%%%%%%%%%%%%%%%%%%%%%%%%%%%%%%%
%%%%%%%%%%%%%%%%%%%%%%%%%%%%%%%%%%%%%%%%%%%%%%%%%%%%%%%%%%%%%%%%%%%%%%%%%%%%%%%%%%%

%%%%%%%%%%%%%%%%%%%%%%%%%%%%%%%%%%%%%%%%%%%%%%%%%%%%%%%%%%%%%%%%%%%%%%%%%%%%%%%%%%%
\subsection{Linear convergence with optimal rates}
%%%%%%%%%%%%%%%%%%%%%%%%%%%%%%%%%%%%%%%%%%%%%%%%%%%%%%%%%%%%%%%%%%%%%%%%%%%%%%%%%%%

Our first main result states linear convergence of the quasi-error product
\begin{equation}
\label{eq:lambda-quasi-error}
	\Lambda_\ell^k
	:=
	\big[ \, \enorm{u_{\ell}^{\star} - u_{\ell}^{k}} + \eta_\ell(u_{\ell}^{k})  \, \big]
	\big[ \, \enorm{z_{\ell}^{\star} - z_{\ell}^{k}} +  \zeta_\ell(z_{\ell}^{k}) \, \big]
	\quad\text{for all }
	(\ell,k) \in \QQ_0
\end{equation}	
for every choice of the stopping parameter $\lctr > 0$.
Recall from \eqref{eq:goal:aposteriori} that the quasi-error product is an upper bound for the error $|G(u^\star) - G_\ell(u_\ell^k,z_\ell^k)|$. 
Moreover, if $k=\k(\ell)$, then  \eqref{eq:contraction} and \eqref{eq:stopping} give that $\Lambda_\ell^\k\simeq \eta_\ell(u_\ell^\k)\zeta_\ell(z_\ell^\k)$.

\begin{theorem}\label{theorem:linear-convergence}
	Suppose \eqref{assumption:stab}--\eqref{assumption:rel}. Suppose that $0 < \theta \le 1$ and $\lctr > 0$.
	Then, Algorithm~\ref{algorithm} satisfies linear convergence in the sense of
	\begin{equation}
	\label{eq:linear-convergence}
		\hspace{-3pt}
		\Lambda_{\ell'}^{k'}
		\leq
		\Clin \qlin^{|(\ell'\!,k')| - |(\ell,k)|} \, \Lambda_{\ell}^{k}
		\quad \text{for all } (\ell,k),(\ell'\!,k') \in \QQ \cup \{(0,0)\}
		\text{ with } |(\ell'\!,k')| \geq |(\ell,k)|.
	\end{equation}
	The constants $\Clin > 0 $ and $0 < \qlin < 1$ depend only on $\Cstab$, $\qred$, $\Crel$, $\qctr$, and the (arbitrary) adaptivity parameters $0 < \theta \le 1$ and $\lctr > 0$.
\end{theorem}

Full linear convergence implies that convergence rates with respect to degrees of freedom and with respect to total computational cost are equivalent.
From this point of view, full linear convergence indeed turns out to be the core argument for optimal complexity.

\begin{corollary}\label{corollary:optimal-convergence}
Recall the definition of the total computational cost $\work(\ell,k)$ from~\eqref{eq:work}.
	Let $r > 0$ and $C_r := \sup_{(\ell,k) \in \QQ} (\#\TT_\ell - \#\TT_0 + 1)^r \Lambda_\ell^k\in[0,\infty]$.
	Then, under the assumptions of Theorem~\ref{theorem:linear-convergence}, it holds that
	\begin{equation}
	\label{eq:rate-equivalence}
		C_r
		\leq
		\sup_{(\ell,k) \in \QQ} (\#\TT_\ell)^r \, \Lambda_\ell^k
		\leq
		\sup_{(\ell,k) \in \QQ} \work(\ell,k)^r \, \Lambda_\ell^k
		\leq
		\Crate \, C_r,
	\end{equation}
	where the constant $\Crate > 0$ depends only on $r$, $\#\TT_0$, and on the constants $\qlin, \Clin$ from Theorem~\ref{theorem:linear-convergence}.
\end{corollary}

\begin{proof}
	The first two estimates in~\eqref{eq:rate-equivalence} are obvious.
	It remains to prove the last estimate in~\eqref{eq:rate-equivalence}.
	To this end, note that it follows from the definition of $C_r$ that
	\begin{equation*}
		\#\TT_\ell - \#\TT_0 + 1
		\leq
		\big( \Lambda_\ell^k \big)^{-1/r} \, C_r^{1/r}
		\quad\text{for all }
		(\ell, k) \in \QQ.
	\end{equation*}
	Moreover, elementary algebra yields that 
	\begin{align*}
	    \#\TT_{\ell'}\le \#\TT_0 (\#\TT_{\ell'}-\#\TT_0+1)
	    \quad\text{for all }(\ell',0)\in\QQ_0.
	\end{align*}
	For $(\ell, k) \in \QQ$, Theorem~\ref{theorem:linear-convergence} and the geometric series thus show that
	\begin{align*}
		&\work(\ell, k)
		\stackrel{\eqref{eq:work}}=
		\hspace{-1em} \sum_{\substack{ (\ell',k') \in \QQ \\ |(\ell',k')| \leq |(\ell,k)| }} \hspace{-1em} \#\TT_{\ell'}
		\leq
		\#\TT_0 \, \hspace{-1em} \sum_{\substack{ (\ell',k') \in \QQ \\ |(\ell',k')| \leq |(\ell,k)| }} \hspace{-1em} (\#\TT_{\ell'} - \#\TT_0 + 1)\\
		&\qquad\leq
		\#\TT_0 C_r^{1/r} \hspace{-1em} \sum_{\substack{ (\ell',k') \in \QQ \\ |(\ell',k')| \leq |(\ell,k)| }} \hspace{-1em} \big( \Lambda_{\ell'}^{k'} \big)^{-1/r}
		\leq
		\#\TT_0 C_r^{1/r} \Clin^{1/r} \frac{1}{1-\qlin^{1/r}} \, \big( \Lambda_{\ell}^{k} \big)^{-1/r}.
	\end{align*}
	With $\Crate := (\#\TT_0)^r \Clin \, 1/(1-\qlin^{1/r})^r$, this gives that
	\begin{equation*}
		\work(\ell, k)^r \Lambda_{\ell}^{k}
		\leq
		\Crate C_r
		\quad\text{for all }
		(\ell, k) \in \QQ.
	\end{equation*}
	This shows the final inequality in~\eqref{eq:rate-equivalence} and thus concludes the proof.
\end{proof}

If $\theta$ and $\lctr$ are small enough, we are able to show that linear convergence from Theorem~\ref{theorem:linear-convergence} even guarantees optimal rates with respect to both the number of unknowns $\#\TT_\ell$ and the total cost $\work(\ell,k)$.
Given $N\in\N_0$, let $\T(N)$ be the set of all $\TT_\coarse\in\T$ with $\#\TT_\coarse-\#\TT_0\le N$.
With 
\begin{subequations}
\begin{align}
    \norm{u^\star}{\mathbb{A}_r} 
    := \sup_{N\in\N_0} (N+1)^r \min_{\TT_{\rm opt}\in\T(N)} \eta_{\rm opt}(u_{\rm opt}^\star) \in [0,\infty]  
\end{align}
and 
\begin{align}
    \norm{z^\star}{\mathbb{A}_r} 
    := \sup_{N\in\N_0} (N+1)^r \min_{\TT_{\rm opt}\in\T(N)}\zeta_{\rm opt}(z_{\rm opt}^\star) \in [0,\infty]  
\end{align}
\end{subequations}
for all $r>0$, there holds the following result.

\begin{theorem}\label{theorem:optimal-convergence}
Recall the definition of the total computational cost $\work(\ell,k)$ from~\eqref{eq:work}.
    Suppose the mesh properties~\eqref{eq:mesh-sons}--\eqref{eq:mesh-overlay} as well as the axioms \eqref{assumption:stab}--\eqref{assumption:drel}. 
    Define
    \begin{align}
        \theta_\star := \frac{1}{1+\Cstab^2\Cdrel^2}
        \quad \text{and} \quad
        \lconv := \frac{1-\qctr}{\qctr \Cstab}. 
    \end{align}
    Let both adaptivity parameters $0<\theta\le1$ and $0<\lctr<\lconv$ be sufficiently small such that 
    \begin{align}\label{eq:small-theta}
        0 < \Big(\frac{\sqrt{2\theta} + \lctr/\lconv}{1-\lctr/\lconv}\Big)^2 < \theta_\star.
    \end{align}
    Let $1 \leq \Cmark < \infty$.
    Suppose that the set of marked elements $\MM_\ell$ in Algorithm~\ref{algorithm}{\rm(iv)} is constructed by one of the strategies from Remark~\ref{rem:marking}{\rm (a)--(c)}, where the sets in~\eqref{eq:bet-marking} and~\eqref{eq:doerfler} have up to the factor $\Cmark$ minimal cardinality. 
    Let $s,t>0$ with $\norm{u^\star}{\mathbb{A}_s}+\norm{z^\star}{\mathbb{A}_t}<\infty$.
    Then, there exists a constant $\Copt>0$ such that 
    \begin{align}
    \label{eq:optimal-rates}
        \sup_{(\ell,k)\in\QQ} \work(\ell,k)^{s+t} \Lambda_\ell^k 
        \le \Copt \max\{\norm{u^\star}{\mathbb{A}_s}\norm{z^\star}{\mathbb{A}_t},\Lambda_0^0\}.
    \end{align}
    The constant $\Copt$ depends only on $\Ccls$, $\Cstab$, $\qred$, $\Crel$, $\Cdrel$, $\qctr$, $\Cmark$, $\theta$, $\lctr$, $\#\TT_0$, $s$, and $t$.
\end{theorem}

\begin{remark}
\label{rem:parameter-constraint}
	The constraint~\eqref{eq:small-theta} is enforced by our analysis of the marking strategy from Remark~\ref{rem:marking}{\rm (a)}, while the marking strategies from Remark~\ref{rem:marking}{\rm (b)--(c)} allow to relax the condition to
	\begin{equation}
	\label{eq:larger-theta}
		0 < \Big(\frac{\sqrt{\theta} + \lctr/\lconv}{1-\lctr/\lconv}\Big)^2 < \theta_\star.
	\end{equation}
\end{remark}

%%%%%%%%%%%%%%%%%%%%%%%%%%%%%%%%%%%%%%%%%%%%%%%%%%%%%%%%%%%%%%%%%%%%%%%%%%%%%%%%%%%
\subsection{Alternative termination criteria for iterative solver}\label{section:termination}
%%%%%%%%%%%%%%%%%%%%%%%%%%%%%%%%%%%%%%%%%%%%%%%%%%%%%%%%%%%%%%%%%%%%%%%%%%%%%%%%%%%
The above formulations of Algorithm~\ref{algorithm} stops the iterative solver for $u_\ell^m$ and the iterative solver for $z_\ell^n$ independently of each other as soon as the respective termination criteria in~\eqref{eq:stopping} are satisfied.
In this section, we briefly discuss two alternative termination criteria:

\textit{Stronger termination:}
The current proof of linear convergence (and of the subsequent proof of optimal convergence) does only exploit that $u_{\ell}^{\k}$ and $z_{\ell}^{\k}$ satisfy the stopping criterion and the previous iterates do not (cf.\ Lemma~\ref{lemma:banach2-new}{\rm(iii)}).
This can also be ensured by the following modification of Algorithm~\ref{algorithm}(i):
\begin{itemize}
	\item[\rm(i)] Employ the iterative solver to compute iterates $u_{\ell}^{1}, \dots, u_{\ell}^{k}$ and $z_{\ell}^{1}, \dots, z_{\ell}^{k}$ together with the corresponding refinement indicators $\eta_\ell(T,u_{\ell}^{k})$ and $\zeta_\ell(T, z_{\ell}^{k})$ for all $T \in \TT_\ell$, until 
	\begin{align}\label{eq:stopping2}
		\enorm{u_{\ell}^{k} - u_{\ell}^{ k-1}} \le \lctr \, \eta_\ell(u_{\ell}^{k})
	 	\quad \text{and} \quad
		\enorm{z_{\ell}^{k} - z_{\ell}^{ k-1}} \le \lctr \, \zeta_\ell(z_{\ell}^{k}).
	\end{align}
\end{itemize}
Note that this will lead to more solver steps, since now $k = \k(\ell)$ (if it exists) is the smallest index for which the stopping criterion holds simultaneously for both $u_{\ell}^{\k}$ and $z_{\ell}^{\k}$. 

Inspecting the proof of Lemma~\ref{lemma:banach2-new} below, we see that all results hold verbatim also for this stopping criterion.
Thus, we conclude linear and optimal convergence (in the sense of Theorem~\ref{theorem:linear-convergence} and Theorem~\ref{theorem:optimal-convergence}) also in this case.

\textit{Natural termination:} The following stopping criterion (which is somehow the most natural candidate) also leads to linear convergence:
Let $\m(\ell), \n(\ell) \in \N$ be minimal with~\eqref{eq:stopping}.
If either of them do not exist, we set again $\m(\ell) = \infty$, or $\n(\ell) = \infty$, respectively.
Define $\k(\ell) := \max\{\m(\ell), \n(\ell)\}$. Then, employ the iterative solver $\k(\ell)$ times for both the primal and the dual problem, i.e., the solver provides iterates $u_\ell^k$ and $z_\ell^k$ until both stopping criteria in~\eqref{eq:stopping} have been satisfied once (which avoids the artificial definition~\eqref{eq2:stopping}).
For instance, if $\m(\ell) < \n(\ell)=\k(\ell)<\infty$, we continue to iterate for the primal problem until $u_{\ell}^{\k}$ is obtained (or never stop the iteration if $\n(\ell)=\k(\ell) = \infty$).
If $\lctr > 0$ is sufficiently small such that $1 - \frac{\qctr}{1-\qctr} \, \Cstab \, (1+\qctr) \lctr > 0$, then we can define
\begin{equation*}
	\lctr
	\leq
	\lctr^\prime
	:=
	\max \Big\{ 1,
	\frac{(1+\qctr)\qctr}{(1-\qctr) \big(
		1 - \frac{\qctr}{1-\qctr} \, \Cstab \, (1+\qctr) \lctr
		\big)} \Big\} \, \lctr
	<
	\infty,
\end{equation*}
and we can guarantee the stopping condition~\eqref{eq:stopping} with the larger constant $\lctr'$, i.e.,
 \begin{equation}
\label{eq:lctr-prime}
	\enorm{u_{\ell}^{ \k} - u_{\ell}^{ \k-1}}
	\leq
	\lctr^\prime \, \eta_\ell(u_{\ell}^{\k})
	\quad\text{and}\quad
	\enorm{z_{\ell}^{ \k} - z_{\ell}^{ \k-1}}
	\leq
	\lctr^\prime \, \zeta_\ell(z_{\ell}^{\k});
\end{equation}
see the proof below.
Again, we notice that then the assumptions of Lemma~\ref{lemma:banach2-new} below are met.
Hence, we conclude linear convergence (in the sense of Theorem~\ref{theorem:linear-convergence}) also for this stopping criterion.
Moreover, optimal rates in the sense of Theorem~\ref{theorem:optimal-convergence} hold if $\lctr$ in~\eqref{eq:small-theta} is replaced by $\lctr'$.

\begin{proof}[Proof of~\eqref{eq:lctr-prime}]
	Without loss of generality, let us assume that $\m(\ell) < \k(\ell) = \n(\ell)<\infty$.
	First, we have that
	\begin{equation*}
		\enorm{u_\ell^\k - u_\ell^\m}
		\leq
		\enorm{u_\ell^\star - u_\ell^\k} + \enorm{u_\ell^\star - u_\ell^\m}
		\leq
		(1 + \qctr^{\k(\ell)-\m(\ell)}) \enorm{u_\ell^\star - u_\ell^\m}.
	\end{equation*}
	Then, using the fact that $u_\ell^\m$ satisfies the stopping criterion in~\eqref{eq:stopping} and stability~\eqref{assumption:stab}, we get that
	\begin{align*}
		\enorm{u_\ell^\star - u_\ell^\m}
		\stackrel{\mathclap{\eqref{eq:contraction}}}{\leq}
		\frac{\qctr}{1-\qctr} \enorm{u_\ell^\m - u_\ell^{\m-1}}
		&\stackrel{\mathclap{\eqref{eq:stopping}}}{\leq}
		\frac{\qctr \lctr}{1-\qctr} \eta_\ell(u_\ell^\m)
		\stackrel{\mathclap{\eqref{assumption:stab}}}{\leq}
		\frac{\qctr \lctr}{1-\qctr} \Big( \eta_\ell(u_\ell^\k) + \Cstab \enorm{u_\ell^\k - u_\ell^\m} \Big) \\
		&\leq
		\frac{\qctr \lctr}{1-\qctr} \Big( \eta_\ell(u_\ell^\k) + \Cstab (1 + \qctr^{\k(\ell)-\m(\ell)}) \enorm{u_\ell^\star - u_\ell^\m} \Big).
	\end{align*}
	For $\lctr < (1-\qctr) / [\Cstab \qctr (1 + \qctr^{\k(\ell)-\m(\ell)})]$ we can absorb the last term to obtain
	\begin{equation*}
		\enorm{u_\ell^\star - u_\ell^\m}
		\leq
		\frac{\qctr}{1-\qctr} \Big( 1 - \frac{\Cstab \qctr}{1-\qctr}(1 + \qctr^{\k(\ell)-\m(\ell)}) \lctr \Big)^{-1} \lctr \eta_\ell(u_\ell^\k).
	\end{equation*}
	Finally, we observe that
	\begin{equation*}
		\enorm{u_{\ell}^{ \k} - u_{\ell}^{ \k-1}}
		\leq
		(1+\qctr) \enorm{u_{\ell}^{ \star} - u_{\ell}^{ \k-1}}
		\leq
		(1+\qctr) \qctr^{\k-\m-1} \enorm{u_{\ell}^{ \star} - u_{\ell}^{ \m}}.
	\end{equation*}
	Combining the last two estimates we obtain that
	\begin{equation*}
		\enorm{u_{\ell}^{ \k} - u_{\ell}^{ \k-1}}
		\le
		\frac{(1+\qctr)\qctr^{\,\k(\ell)-\m(\ell)}}{(1-\qctr) \big(
			1 - \frac{\qctr}{1-\qctr} \, \Cstab \, (1+\qctr^{\,\k(\ell)-\m(\ell)}) \lctr
		\big)} \,
		\lctr \, \eta_\ell(u_{\ell}^{\k}).
	\end{equation*}
	Hence, \eqref{eq:lctr-prime} follows with $\qctr^{\,\k(\ell)-\m(\ell)} \leq \qctr$ and $\enorm{z_{\ell}^{ \k} - z_{\ell}^{ \k-1}} \leq	\lctr \, \zeta_\ell(z_{\ell}^{\k}) \leq \lctr^\prime \, \zeta_\ell(z_{\ell}^{\k})$.
\end{proof}

%\clearpage
%!TEX root = main.tex

%%%%%%%%%%%%%%%%%%%%%%%%%%%%%%%%%%%%%%%%%%%%%%%%%%%%%%%%%%%%%%%%%%%%%%%%%%%%%%%%%%%
%%%%%%%%%%%%%%%%%%%%%%%%%%%%%%%%%%%%%%%%%%%%%%%%%%%%%%%%%%%%%%%%%%%%%%%%%%%%%%%%%%%
\section{Numerical examples}\label{sec:numerics}
%%%%%%%%%%%%%%%%%%%%%%%%%%%%%%%%%%%%%%%%%%%%%%%%%%%%%%%%%%%%%%%%%%%%%%%%%%%%%%%%%%%
%%%%%%%%%%%%%%%%%%%%%%%%%%%%%%%%%%%%%%%%%%%%%%%%%%%%%%%%%%%%%%%%%%%%%%%%%%%%%%%%%%%

In this section, we consider two numerical examples which solve the equation 
\begin{equation}
\label{eq:numerics-problem}
\begin{split}
	- \Delta u^\star &= f \hphantom{\phi 0} \quad \text{in } \Omega,\\
	u^\star &= 0 \hphantom{\phi f} \quad \text{on } \Gamma_D,\\
	\nabla u^\star \cdot \vec{n} &= \phi \hphantom{f 0} \quad \text{on } \Gamma_N,
\end{split}
\end{equation}
where $\phi \in L^2(\Gamma_N)$ and $\normalvec$ is the element-wise outwards facing unit normal vector.
We refer the reader to Remark~\ref{rem:neumann-boundary} for a comment on the applicability of our results to this model problem.
We further suppose that the goal functional is a slight variant of the one proposed in~\cite{ms2009}, i.e.,
\begin{equation}
\label{eq:numerical-goal}
	G(v) = - \int_{\omega} \nabla v \cdot \g \d{x}
	\qquad
	\text{for } v \in H^1_D(\Omega),
\end{equation}
with a subset $\omega \subseteq \Omega$ and a fixed direction $\g(x) = \g_0 \in \R^2$.
Moreover, for error estimation, we employ standard residual error estimators, which in our case, for all $(\ell, k) \in \QQ$ and all $T \in \TT_\ell$, read
\begin{align*}
	\eta_\ell(T, u_\ell^k)^2
	&:=
	h_T^2 \norm{\Delta u_\ell^k + f}{L^2(T)}^2
	+ h_T \norm{\jump{\nabla u_\ell^k \cdot \normalvec}}{L^2(\partial T \cap \Omega)}^2
	+ h_T \norm{\nabla u_\ell^k \cdot \normalvec - \phi}{L^2(\partial T \cap \Gamma_N)}^2,\\
	\zeta_\ell(T, z_\ell^k)^2
	&:=
	h_T^2 \norm{\div(\nabla z_\ell^k + \g)}{L^2(T)}^2
	+ h_T \norm{\jump{(\nabla z_\ell^k + \g) \cdot \normalvec}}{L^2(\partial T \cap \Omega)}^2,
\end{align*}
where $h_T = |T|^{1/2}$ is the local mesh-width and $\jump{\cdot}$ denotes the jump across interior edges.
It is well-known \cite{axioms,fpz2014} that $\eta_\ell$ and $\zeta_\ell$ satisfy the assumptions~\eqref{assumption:stab}--\eqref{assumption:drel}.
The examples are chosen to showcase the performance of the proposed GOAFEM algorithm for different types of singularities.

Throughout this section, we solve~\eqref{eq:numerics-problem} as well as the corresponding dual problem numerically using Algorithm~\ref{algorithm}, where we make the following choices:
\begin{itemize}
	\item We solve the problems on the lowest order finite element space, i.e., with polynomial degree $p=1$.
	
	\item As initial values, we use $u_0^0 = z_0^0 = 0$.
	
	\item To solve the arising linear systems, we use a preconditioned conjugate gradient (PCG) method with an optimal additive Schwarz preconditioner.
	We refer to~\cite{cnx12,schimanko} for details and, in particular, the proof that this iterative solver satisfies~\eqref{eq:solver}.
	
	\item We use the marking criterion from Remark~\ref{rem:marking}{\rm(a)} and choose $\MM_\ell$ such that it has minimal cardinality.
	
	\item Unless mentioned otherwise, we use $\vartheta = 0.5$ and $\lctr = 10^{-5}$.
\end{itemize}

%%%%%%%%%%%%%%%%%%%%%%%%%%%%%%%%%%%%%%%%%%%%%%%%%%%%%%%%%%%%%%%%%%%%%%%%%%%%%%%%%%%
\subsection{Singularity in goal functional only}\label{subsec:example1}
%%%%%%%%%%%%%%%%%%%%%%%%%%%%%%%%%%%%%%%%%%%%%%%%%%%%%%%%%%%%%%%%%%%%%%%%%%%%%%%%%%%

\begin{figure}
	\centering
	\includegraphics[width=0.3\linewidth]{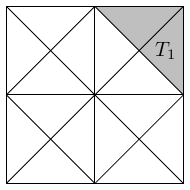}
	\qquad\qquad
	\includegraphics[width=0.3\linewidth]{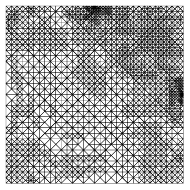}
	\caption{Left: Initial mesh $\TT_0$. The shaded area is the set $T_1$ from Section~\eqref{subsec:example1}.
		Right: Mesh after $14$ iterations of Algorithm~\ref{algorithm} with $\#\TT_{14} = 4157$.}
	\label{fig:mesh}
\end{figure}

In our first example, the primal problem is~\eqref{eq:numerics-problem} with $f = 2x_1(1-x_1) + 2x_2(1-x_2)$ on the unit square $\Omega = (0,1)^2$, and $\Gamma_D = \partial \Omega$ (and thus, $\Gamma_N = \emptyset$).
For this problem, the exact solution reads
\begin{equation*}
	u^\star(x) =x_1 x_2 (1-x_1) (1-x_2).
\end{equation*}
The goal functional is~\eqref{eq:numerical-goal} with $\omega = T_1 := \set{x \in \Omega}{x_1 + x_2 \geq 3/2}$ and $\g_0 = (-1, 0)$.
The exact goal value can be computed analytically to be 
\begin{equation*}
	G(u^\star) = \int_{T_1} \frac{\partial u^\star}{\partial x_1} \d{x}  = 11/960.
\end{equation*}
The initial mesh $\TT_0$ as well as a visualization of the set $T_1$ can be seen in Figure~\ref{fig:mesh}.

\begin{figure}
	\centering
	\includegraphics[width=0.7\linewidth]{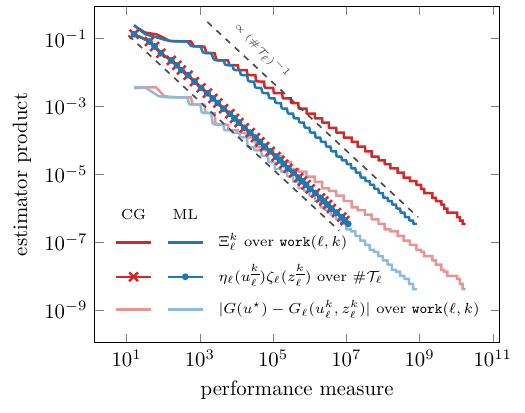}
	\caption{Comparison between iterative solvers for the problem from Section~\ref{subsec:example1}.
		A conjugate gradient method without preconditioner (CG) leads to optimal rates with respect to $\#\TT_\ell$ for the final iterates where $k = \k(\ell)$, but not with respect to $\work(\ell,k)$ for every $(\ell, k) \in \QQ$.
		Our choice of the iterative solver (ML) achieves optimal rates with respect to both measures.}
	\label{fig:cumulativecosts}
\end{figure}

For this setting, we compare our iterative solver to a conjugate gradient method without preconditioner in Figure~\ref{fig:cumulativecosts}, where we plot the computable upper bound from~\eqref{eq2:goal:aposteriori},
\begin{equation*}
	\Xi_\ell^k
	:=
	\big[ \eta_\ell(u_\ell^k) + \enorm{u_\ell^k - u_\ell^{k-1}} \big]
	\big[ \zeta_\ell(z_\ell^k) + \enorm{z_\ell^k - z_\ell^{k-1}} \big]
	\quad
	\text{for all } (\ell,k) \in \QQ,
\end{equation*}
over $\work(\ell, k)$ for all iterates $(\ell,k) \in \QQ$ and the estimator product for the final iterates $\eta_\ell(u_\ell^\k) \zeta_\ell(z_\ell^\k)$ over $\# \TT_\ell$.
We stress that, for $(\ell,k) \in \QQ$, the computable upper bound $\Xi_\ell^k$  and the quasi-error product $\Lambda_\ell^k$ from~\eqref{eq:lambda-quasi-error} are related by $\Lambda_\ell^k \lesssim \Xi_\ell^k \lesssim \Lambda_\ell^{k-1}$ so that linear convergence~\eqref{eq:linear-convergence} with optimal rates~\eqref{eq:optimal-rates} of $\Lambda_\ell^k$ also yields linear convergence with optimal rates of $\Xi_\ell^k$.
Since in our experiments $\lctr = 10^{-5}$ is small, it is plausible to assume that the final estimates on every level approximate the exact solutions sufficiently well in the sense of estimator products, i.e., $\eta_\ell(u_\ell^\k)\zeta_\ell(z_\ell^\k) \approx \eta_\ell(u_\ell^\star)\zeta_\ell(z_\ell^\star)$ (cf.\ Lemma~\ref{lemma:banach} below) for which \cite{fpz2014} proves optimal convergence rates with respect to $\#\TT_\ell$.
Indeed, we see optimal rates for $\eta_\ell(u_\ell^\k) \zeta_\ell(z_\ell^\k)$ with respect to $\#\TT_\ell$ for both solvers in Figure~\ref{fig:cumulativecosts}.
However, the non-preconditioned CG method fails to satisfy uniform contraction~\eqref{eq:solver} and thus Theorem~\ref{theorem:optimal-convergence} cannot be applied.
In fact, Figure~\ref{fig:cumulativecosts} shows that this method fails to drive down $\Xi_\ell^k$ with optimal rates with respect to $\work(\ell,k)$ (cf.~\eqref{eq:work}), as opposed to the optimally preconditioned PCG method.

\begin{figure}
	\centering
	\includegraphics[width=0.99\linewidth]{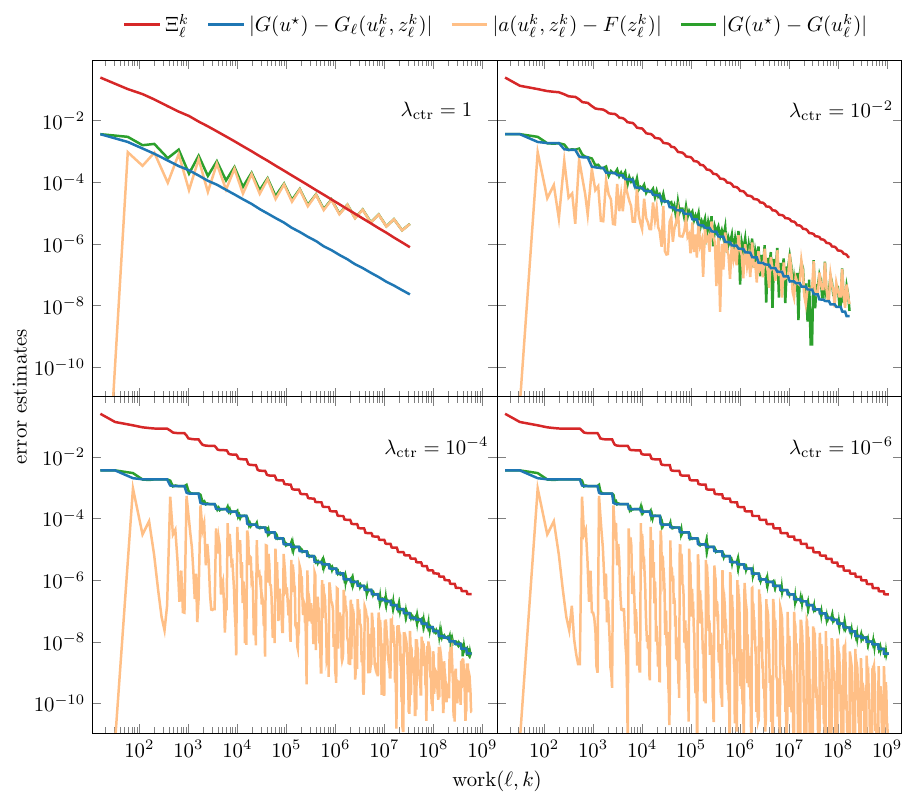}
	\caption{Comparison between $\Xi_\ell^k$, discrete goal $G_\ell(u_\ell^k, z_\ell^k)$, primal residual evaluated at the dual solution $z_\ell^k$, and direct evaluation of goal functional $G(u_\ell^k)$ for every iterate $(\ell,k) \in \QQ$ and different values of $\lctr \in \{ 1, 10^{-2}, 10^{-4}, 10^{-6} \}$.
	The primal residual evaluated at the dual solution $z_\ell^k$ is the difference between goal and discrete goal; see~\eqref{eq:def:goal}.}
	\label{fig:solverresidual}
\end{figure}

Furthermore, we plot in Figure~\ref{fig:solverresidual} different error measures over $\work(\ell,k)$ for every iterate $(\ell,k) \in \QQ$.
This shows that the corrector term
\begin{equation}
	a(u_\ell^k , z_\ell^k) - F(z_\ell^k)
\end{equation}
(which is the residual of $u_\ell^k$ evaluated at the dual solution $z_\ell^k$) in the definition of the discrete goal functional~\eqref{eq:def:goal} is indeed necessary.
We see that throughout the iteration, the goal value $G(u_\ell^k)$ highly oscillates and, for large values of $\lctr$, even shows a different rate than the $\Xi_\ell^k$ over $\work(\ell,k)$.
In general, we thus cannot expect the quantity $\Xi_\ell^k$ to bound the uncorrected goal-error $|G(u^\star) - G(u_\ell^k)|$.

For the discrete goal, the corrector term compensates the oscillations of the goal functional, such that their sum decreases with the same rate as $\Xi_\ell^k$, as predicted by~\eqref{eq2:goal:aposteriori}.
Smaller values of $\lctr$ imply that on every level $\ell$ the approximate solutions $u_\ell^k, z_\ell^k$ are computed more accurately, such that the corrector term becomes smaller and the effect on the rate of the goal value becomes negligible.

%%%%%%%%%%%%%%%%%%%%%%%%%%%%%%%%%%%%%%%%%%%%%%%%%%%%%%%%%%%%%%%%%%%%%%%%%%%%%%%%%%%
\subsection{Geometrical singularity}\label{subsec:example2}
%%%%%%%%%%%%%%%%%%%%%%%%%%%%%%%%%%%%%%%%%%%%%%%%%%%%%%%%%%%%%%%%%%%%%%%%%%%%%%%%%%%

Our second example is the classical example of a geometric singularity on the so-called Z-shape $\Omega = (-1,1)^2 \setminus \mathrm{conv}\{(-1,-1), (0,0), (-1, 0)\}$, where $\Gamma_D$ is only the re-entrant corner (cf.\ Figure~\ref{fig:geometry-mesh}).
The primal problem is~\eqref{eq:numerics-problem} with $f=0$ and $\phi = \nabla u^\star \cdot \vec{n}$, where the exact solution in polar coordinates $r(x)$ and $\varphi(x)$ of $x \in \R^2$ is prescribed as
\begin{equation*}
	u^\star(x)
	=
	r(x)^{4/7} \sin(\tfrac{4}{7} \varphi(x) + \tfrac{3\pi}{7}).
\end{equation*}
The goal functional is~\eqref{eq:numerical-goal} with $\omega = T_2 := (0.5, 0.5)^2 \cap \Omega$ and $\g_0 = (-1,-1)$ and can be computed directly via numerical integration to be
\begin{equation*}
	G(u^\star)
	=
	\int_{T_2} \Big( \frac{\partial u^\star}{\partial x_1} + \frac{\partial u^\star}{\partial x_2} \Big) \d{x}
	\approx
	0.82962247157810.
\end{equation*}
In Figure~\ref{fig:geometry-mesh}, the initial triangulation $\TT_0$ as well as the mesh after several iterations of Algorithm~\ref{algorithm} can be seen.
The adaptive algorithm resolves the singularity at the re-entrant corner, as well as critical points of the goal functional, which are at the corners of $T_2$.

\begin{figure}
	\centering
	\includegraphics[width=0.3\linewidth]{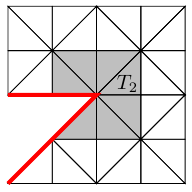}
	\qquad \qquad
	\includegraphics[width=0.3\linewidth]{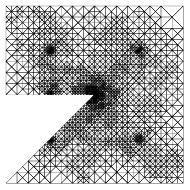}
	\caption{Left: Initial mesh $\TT_0$. The shaded area is the set $T_2$ from Section~\eqref{subsec:example2} and the Dirichlet boundary at the re-entrant corner is marked in red.
	Right: Mesh after $13$ iterations of Algorithm~\ref{algorithm} with $\#\TT_{13} = 4534$.}
	\label{fig:geometry-mesh}
\end{figure}

Figure~\ref{fig:geometry-cumulativecosts} shows the rate of the estimator product $\eta_\ell(u_\ell^\k) \zeta_\ell(z_\ell^\k)$ of the final iterates over $\#\TT_\ell$ as well as the rate of $\Xi_\ell^k$ over $\work(\ell, k)$ for all $(\ell,k) \in \QQ$.

\begin{figure}
	\centering
	\includegraphics[width=0.7\linewidth]{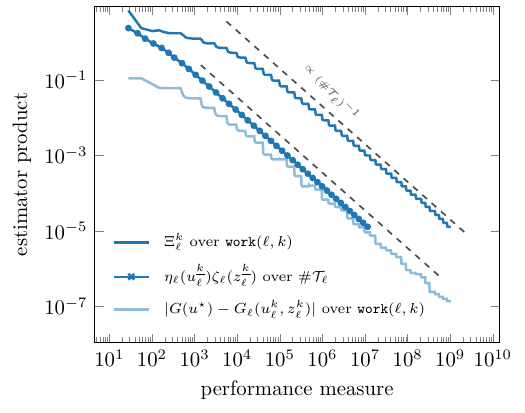}
	\caption{Rates of the estimator product for final iterates over $\#\TT_\ell$ and $\Xi_\ell^k$ as well as goal error over $\work(\ell,k)$ for all $(\ell,k) \in \QQ$.}
	\label{fig:geometry-cumulativecosts}
\end{figure}
%\clearpage
%!TEX root = main.tex

%%%%%%%%%%%%%%%%%%%%%%%%%%%%%%%%%%%%%%%%%%%%%%%%%%%%%%%%%%%%%%%%%%%%%%%%%%%%%%%%%%%
%%%%%%%%%%%%%%%%%%%%%%%%%%%%%%%%%%%%%%%%%%%%%%%%%%%%%%%%%%%%%%%%%%%%%%%%%%%%%%%%%%%
\section{Proof of Theorem~\ref{theorem:linear-convergence}}\label{sec:proof1}
%%%%%%%%%%%%%%%%%%%%%%%%%%%%%%%%%%%%%%%%%%%%%%%%%%%%%%%%%%%%%%%%%%%%%%%%%%%%%%%%%%%
%%%%%%%%%%%%%%%%%%%%%%%%%%%%%%%%%%%%%%%%%%%%%%%%%%%%%%%%%%%%%%%%%%%%%%%%%%%%%%%%%%%

The following core lemma extends one of the key observations of~\cite{banach2} to the present setting, where we stress that the nonlinear product structure of $\Delta_\ell^k$ leads to technical challenges which go much beyond \cite{banach2}.

\begin{lemma}\label{lemma:banach2-new}
	Suppose \eqref{assumption:stab}--\eqref{assumption:rel}.
	Then, there exist constants $\mu, \Caux > 0$, and $0 < \qaux < 1$, and some scalar sequence $(R_\ell)_{\ell \in \N_0} \subset \R$ such that the quasi-error product
	\begin{equation*}
	\Delta_{\ell}^k
	:=
	\big[ \, \enorm{u_{\ell}^{\star} - u_{\ell}^{k}} + \mu \, \eta_\ell(u_{\ell}^{k})  \, \big]
	\big[ \, \enorm{z_
		{\ell}^{\star} - z_{\ell}^{k}} + \mu \, \zeta_\ell(z_{\ell}^{k}) \, \big]
	\quad\text{for all } (\ell,k) \in \QQ_0
	\end{equation*}
	satisfies the following statements~{\rm(i)}--{\rm(v)}:
	\begin{itemize}
		\item[\rm(i)] $\Delta_\ell^k \leq \Delta_\ell^j$
		\quad for all $0 \leq j \leq k \leq \k(\ell)$.
		\item[\rm(ii)] $\Delta_\ell^{\k-1} \leq \Caux \, \Delta_\ell^\k$ \quad if $\k(\ell) < \infty$.
		\item[\rm(iii)] $\Delta_{\ell}^{k} \le \qaux \, \Delta_{\ell}^{k-1}$
		\quad for all $0 < k < \k(\ell)$.
		\item[\rm(iv)] $\Delta_{\ell+1}^{0} \le \qaux \, \Delta_{\ell}^{\k-1} + R_\ell$
		\quad for all $0 < \ell < \underline{\ell}$.
		\item[\rm(v)] $\sum_{\ell = \ell'}^{\underline{\ell}-1} R_{\ell}^2 \leq \Caux (\Delta_{\ell}^{\k-1})^2$
		\quad for all $0 \leq \ell' < \underline{\ell}-1$.
	\end{itemize}
	The constants $\mu$, $\Caux$, and $\qaux$ depend only on $\Cstab$, $\qred$, $\Crel$, and $\qctr$ as well as on the (arbitrary) adaptivity parameters $0 < \theta \le 1$ and $\lctr > 0$.
\end{lemma}

For the following proofs, we define
\begin{align*}
\alpha_\ell^k
&:=
\enorm{u_\ell^\star - u_{\ell}^{k}},
&x_\ell^\star
&:=
\enorm{u_{\ell+1}^\star - u_{\ell}^{\star}},\\
\beta_\ell^k
&:=
\enorm{z_\ell^\star - z_{\ell}^{k}},
&y_\ell^\star
&:=
\enorm{z_{\ell+1}^\star -z_{\ell}^{\star}},
\end{align*}
such that the quasi-error product reads $\Delta_{\ell}^{k} = \big[ \alpha_{\ell}^{k} + \mu \, \eta_\ell(u_{\ell}^{k}) \big] \big[ \beta_{\ell}^{k} + \mu \, \zeta_\ell(z_{\ell}^{k}) \big]$ with a free parameter $\mu > 0$ which will be fixed below.

\begin{proof}[Proof of Lemma~\ref{lemma:banach2-new}\rm(i)]
	Recall from~\eqref{eq2:stopping} that $u_\ell^k = u_\ell^\m$ for all $\m(\ell) < k \leq \k(\ell)$.
		Thus, we have that
		\begin{equation*}
		\alpha_\ell^k + \mu \, \eta_\ell(u_{\ell}^{k})
		=
		\alpha_\ell^\m + \mu \, \eta_\ell(u_{\ell}^{\m})
		\quad \text{for all }
		\m(\ell) < k \leq \k(\ell).
		\end{equation*}
	For $0 < k < \m(\ell)$, on the other hand, the solution $u_\ell^k$ is obtained by one step of the iterative solver.
	From stability~\eqref{assumption:stab} and solver contraction~\eqref{eq:solver}, we have for all $0 \leq j < k \leq \m(\ell)$ that
	\begin{align*}
	\alpha_\ell^k + \mu \, \eta_\ell(u_{\ell}^{k})
	&\stackrel{\mathclap{\eqref{assumption:stab}}}{\leq}
	\alpha_\ell^k + \mu \, \big[ \eta_\ell(u_{\ell}^{j}) + \Cstab \enorm{u_{\ell}^{k} - u_{\ell}^{j}} \big]\\
	&\stackrel{\mathclap{\eqref{eq:solver}}}{\leq}
	\big( \qctr^{k-j} + \mu \Cstab (1+\qctr^{k-j}) \big) \alpha_\ell^j + \mu \, \eta_\ell(u_{\ell}^{j})
	\leq
	\big( \qctr + 2 \mu \Cstab \big) \alpha_\ell^j + \mu \, \eta_\ell(u_{\ell}^{j}).
	\end{align*}
	If $\mu$ is chosen small enough such that $\qctr + 2 \mu \Cstab \leq 1$, together with the trivial case $j = k$, the last two equations show that
		\begin{equation*}
		\alpha_\ell^k + \mu \, \eta_\ell(u_{\ell}^{k})
		\leq
		\alpha_\ell^j + \mu \, \eta_\ell(u_{\ell}^{j})
		\quad \text{for all }
		0 \leq j \leq k \leq \k(\ell).
		\end{equation*}
		The same argument shows that
	\begin{equation}
	\label{eq:beta-monotonicity}
	\beta_\ell^k + \mu \, \zeta_\ell(z_{\ell}^{k})
	\leq
	\beta_\ell^j + \mu \, \zeta_\ell(z_{\ell}^{j}).
		\quad \text{for all }
		0 \leq j \leq k \leq \k(\ell).
	\end{equation}
	Multiplication of the last two estimates shows the assertion.
\end{proof}

\begin{proof}[Proof of Lemma~\ref{lemma:banach2-new}\rm(ii)]
	Recall that for the index $\k(\ell)$ there holds~\eqref{eq:stopping}.
	From the triangle inequality, we thus get for the primal estimator that
	\begin{equation*}
	\alpha_\ell^{\k-1}
	=
	\enorm{u_\ell^\star - u_\ell^{\k-1}}
	\leq
	\enorm{u_\ell^\star - u_\ell^{\k}} + \enorm{u_\ell^\k - u_\ell^{\k-1}}
	\stackrel{\mathclap{\eqref{eq:stopping}}}{\leq}
	\alpha_\ell^{\k} + \lctr \, \eta_\ell(u_\ell^\k).
	\end{equation*}
	Furthermore, stability~\eqref{assumption:stab} leads to
	\begin{equation*}
	\eta_\ell(u_\ell^{\k-1})
	\stackrel{\mathclap{\eqref{assumption:stab}}}{\leq}
	\eta_\ell(u_\ell^\k) + \Cstab \,\enorm{u_\ell^\k - u_\ell^{\k-1}}
	\stackrel{\mathclap{\eqref{eq:stopping}}}{\leq}
	\big( 1 + \lctr \Cstab \big) \eta_\ell(u_\ell^\k).
	\end{equation*}
	Combining the last two estimates, we see that
	\begin{equation*}
	\alpha_\ell^{\k-1} + \mu \, \eta_\ell(u_\ell^{\k-1})
	\leq
	\big(  1 + \lctr (\Cstab + \mu^{-1}) \big) \,
	\big[ \, \alpha_\ell^{\k} + \mu \, \eta_\ell(u_\ell^{\k}) \, \big].
	\end{equation*}
	Together with the analogous estimate for $\beta_\ell^{\k-1} + \mu \, \zeta_\ell(z_\ell^{\k-1})$, we conclude the proof with $\Caux = \big(  1 + \lctr (\Cstab + \mu^{-1}) \big)^2$.
\end{proof}

\begin{proof}[Proof of Lemma~\ref{lemma:banach2-new}\rm(iii)]
	Without loss of generality, suppose that $\k(\ell) = \m(\ell)$ and thus $\enorm{u_{\ell}^{k} - u_{\ell}^{k-1}} > \lctr \, \eta_\ell(u_{\ell}^{k})$.
	Then, this yields that
	\begin{equation*}
	\eta_\ell(u_{\ell}^{k})
	<
	\lctr^{-1} \, \enorm{u_{\ell}^{k} - u_{\ell}^{k-1}}
	\stackrel{\mathclap{\eqref{eq:contraction}}}{\leq}
	\lctr^{-1} \, (1+\qctr) \, \alpha_{\ell}^{k-1}
	\quad\text{for all }
	0 < k < \k(\ell).
	\end{equation*}
	With contraction of the solver~\eqref{eq:solver}, this leads to
	\begin{equation*}
	\alpha_{\ell}^k + \mu \, \eta_\ell(u_{\ell}^{k})
	\leq
	\qctr \alpha_{\ell}^{k-1} + \mu \lctr^{-1} (1+\qctr) \, \alpha_{\ell}^{k-1}
	\quad\text{for all }
	0 < k < \k(\ell).
	\end{equation*}
	From~\eqref{eq:beta-monotonicity} for $\mu$ small enough, we see that $\beta_\ell^k + \mu \, \zeta_\ell(z_{\ell}^{k}) \leq \beta_\ell^{k-1} + \mu \, \zeta_\ell(z_{\ell}^{k-1})$.
	Together with the previous estimate, this shows that
	\begin{equation}
	\label{eq:Lambda-estimate1}
	\Delta_\ell^k
	\leq
	\big( \qctr + \mu \lctr^{-1} (1+\qctr) \big) \Delta_\ell^{k-1}.
	\end{equation}
	Up to the choice of $\mu$, this concludes the proof.
\end{proof}

\begin{proof}[Proof of Lemma~\ref{lemma:banach2-new}\rm(iv)]
	First, we note that $\eta_\ell(u_\ell^\k) \zeta_\ell(z_\ell^\k) \neq 0$, according to Algorithm~\ref{algorithm}(iii) and the assumption that $\ell < \underline{\ell}$.
	From reduction of the solver~\eqref{eq:solver} and nested iteration, we get that
	\begin{equation}\label{eq:alpha-beta}
	\begin{split}
	\alpha_{\ell+1}^0
	&=
	\enorm{u_{\ell+1}^{\star} - u_\ell^\k}
	\leq
	\enorm{u_{\ell+1}^{\star} - u_{\ell}^{\star}} + \qctr \, \enorm{u_\ell^\star - u_{\ell}^{\k-1}}
	=
	x_\ell^\star + \qctr \, \alpha_\ell^{\k-1},\\
	\beta_{\ell+1}^0
	&=
	\enorm{z_{\ell+1}^{\star} - z_\ell^\k}
	\leq
	\enorm{z_{\ell+1}^{\star} - z_{\ell}^{\star}} + \qctr \, \enorm{z_\ell^\star - z_{\ell}^{\k-1}}
	=
	y_\ell^\star + \qctr \, \beta_\ell^{\k-1}
	\end{split}
	\end{equation}
	and thus
	\begin{equation}\label{eq:result-1}
	\alpha_{\ell+1}^0 \beta_{\ell+1}^0
	\leq
	\qctr^2 \, \alpha_{\ell}^{\k-1} \beta_{\ell}^{\k-1} + \qctr ( \alpha_{\ell}^{\k-1} y_{\ell}^\star + \beta_{\ell}^{\k-1} x_{\ell}^\star) + x_{\ell}^\star y_{\ell}^\star.
	\end{equation}
	For the estimator terms, we have with stability~\eqref{assumption:stab} and reduction~\eqref{assumption:red} that
	\begin{align*}
	\eta_{\ell+1}(u_{\ell+1}^0)^2
	=
	\eta_{\ell+1}(u_{\ell}^\k)^2
	&=
	\eta_{\ell+1}(\TT_{\ell+1} \cap \TT_\ell, u_{\ell}^\k)^2
	+ \eta_{\ell+1}(\TT_{\ell+1} \setminus \TT_\ell, u_{\ell}^\k)^2\\
	&\leq
	\eta_{\ell}(\TT_{\ell+1} \cap \TT_\ell, u_{\ell}^\k)^2
	+ \qred^2 \, \eta_{\ell}(\TT_{\ell} \setminus \TT_{\ell+1}, u_{\ell}^\k)^2\\
	&=
	\eta_{\ell}(u_{\ell}^\k)^2
	- (1-\qred^2) \, \eta_{\ell}(\TT_{\ell} \setminus \TT_{\ell+1}, u_{\ell}^\k)^2.
	\end{align*}
	On the one hand, with $C_1 := \Cstab (1+\qred)$, this implies that
	\begin{equation}\label{eq:marking-stability}
	\eta_{\ell+1}(u_{\ell+1}^0)
	\leq
	\eta_{\ell}(u_{\ell}^\k)
	\stackrel{\mathclap{\eqref{assumption:stab}}}{\leq}
	\eta_{\ell}(u_{\ell}^{\k-1}) + \Cstab \enorm{u_{\ell}^{\k} - u_{\ell}^{\k-1}}
	\stackrel{\mathclap{\eqref{eq:contraction}}}{\leq}
	\eta_{\ell}(u_{\ell}^{\k-1}) + C_1 \, \alpha_{\ell}^{\k-1}.
	\end{equation}
	On the other hand, with $0 < q_\theta := 1-(1-\qred^2)\theta < 1$, we get that
	\begin{equation}\label{eq:marking-reduction}
	\frac{\eta_{\ell+1}(u_{\ell+1}^0)^2}{\eta_{\ell}(u_{\ell}^\k)^2}
	\leq
	q_\theta + (1-\qred^2) \Big[ \theta - \frac{\eta_{\ell}(\TT_{\ell} \setminus \TT_{\ell+1}, u_{\ell}^\k)^2}{\eta_{\ell}(u_{\ell}^\k)^2} \Big].
	\end{equation}
	Using~\eqref{eq:marking-reduction}, the corresponding estimate for the dual estimator, and the Young inequality, we obtain that
	\begin{equation*}
	\frac{\eta_{\ell+1}(u_{\ell+1}^0)}{\eta_{\ell}(u_{\ell}^\k)}
	\frac{\zeta_{\ell+1}(z_{\ell+1}^0)}{\zeta_{\ell}(z_{\ell}^\k)}
	\leq
	q_\theta
	+ \frac{(1-\qred^2)}{2} \Big[
	2\theta
	- \frac{\eta_{\ell}(\TT_{\ell} \setminus \TT_{\ell+1}, u_{\ell}^\k)^2}{\eta_{\ell}(u_{\ell}^\k)^2}
	- \frac{\zeta_{\ell}(\TT_{\ell} \setminus \TT_{\ell+1}, z_{\ell}^\k)^2}{\zeta_{\ell}(z_{\ell}^\k)^2}
	\Big].
	\end{equation*}
	The marking criterion~\eqref{eq:abstract-marking}, which is applicable due to $\ell < \underline{\ell}$, estimates the term in brackets by zero.
	Thus stability~\eqref{assumption:stab} leads to
	\begin{equation}\label{eq:result-2}
	\begin{split}
	&\eta_{\ell+1}(u_{\ell+1}^0) \zeta_{\ell+1}(z_{\ell+1}^0)
	\leq
	q_\theta \, \eta_{\ell}(u_{\ell}^\k) \zeta_{\ell}(z_{\ell}^\k)\\
	&\qquad \stackrel{\mathclap{\eqref{assumption:stab}}}{\leq}
	q_\theta \,
	\big[ \eta_{\ell}(u_{\ell}^{\k-1}) + \Cstab \enorm{u_{\ell}^{\k} - u_{\ell}^{\k-1}}\big]
	\big[ \zeta_{\ell}(z_{\ell}^{\k-1}) + \Cstab \enorm{z_{\ell}^{\k} - z_{\ell}^{\k-1}}\big]\\
	&\qquad \stackrel{\mathclap{\eqref{eq:contraction}}}{\leq}
	q_\theta \, \eta_{\ell}(u_{\ell}^{\k-1}) \zeta_{\ell}(z_{\ell}^{\k-1})
	+ q_\theta C_1 \big[
	\eta_{\ell}(u_{\ell}^{\k-1}) \beta_{\ell}^{\k-1}
	+ \zeta_{\ell}(z_{\ell}^{\k-1}) \alpha_{\ell}^{\k-1}
	\big]
	+ C_1^2 \, \alpha_{\ell}^{\k-1} \beta_{\ell}^{\k-1}.
	\end{split}
	\end{equation}
	For the mixed terms in $\Delta_{\ell+1}^0$, we have with~\eqref{eq:alpha-beta} and \eqref{eq:marking-stability} that
	\begin{equation}\label{eq:result-3}
	\begin{split}
	\eta_{\ell+1}(u_{\ell+1}^0) \beta_{\ell+1}^0
	&\leq
	\big[ \eta_{\ell}(u_{\ell}^{\k-1}) + C_1 \, \alpha_{\ell}^{\k-1} \big]
	\big[ y_\ell^\star + \qctr \, \beta_\ell^{\k-1} \big]\\
	&=
	\qctr \, \eta_{\ell}(u_{\ell}^{\k-1}) \beta_\ell^{\k-1}
	+ \eta_{\ell}(u_{\ell}^{\k-1}) y_\ell^\star
	+ C_1 \, \alpha_\ell^{\k-1} y_\ell^\star
	+ C_1 \qctr \, \alpha_\ell^{\k-1} \beta_\ell^{\k-1}.
	\end{split}
	\end{equation}
	Analogously, we see that
	\begin{equation}\label{eq:result-4}
	\zeta_{\ell+1}(z_{\ell+1}^0) \alpha_{\ell+1}^0
	\leq
	\qctr \, \zeta_{\ell}(z_{\ell}^{\k-1}) \alpha_\ell^{\k-1}
	+ \zeta_{\ell}(z_{\ell}^{\k-1}) x_\ell^\star
	+ C_1 \, \beta_\ell^{\k-1} x_\ell^\star
	+ C_1 \qctr \, \alpha_\ell^{\k-1} \beta_\ell^{\k-1}.
	\end{equation}
	Combining~\eqref{eq:result-1} and \eqref{eq:result-2}--\eqref{eq:result-4}, we get that
	\begin{align*}
	\Delta_{\ell+1}^0
	&=
	\alpha_{\ell+1}^0 \beta_{\ell+1}^0
	+ \mu \, \big[
	\eta_{\ell+1}(u_{\ell+1}^0) \beta_{\ell+1}^0
	+ \zeta_{\ell+1}(z_{\ell+1}^0) \alpha_{\ell+1}^0
	\big]
	+ \mu^2 \, \eta_{\ell+1}(u_{\ell+1}^0) \zeta_{\ell+1}(z_{\ell+1}^0)\\
	&\leq
	\qctr^2 \, \alpha_{\ell}^{\k-1} \beta_{\ell}^{\k-1}
	+ \qctr ( \alpha_{\ell}^{\k-1} y_{\ell}^\star + \beta_{\ell}^{\k-1} x_{\ell}^\star) + x_{\ell}^\star y_{\ell}^\star\\
	&\quad
	+ \mu \, \big[
	\qctr \, \eta_{\ell}(u_{\ell}^{\k-1}) \beta_\ell^{\k-1}
	+ \eta_{\ell}(u_{\ell}^{\k-1}) y_\ell^\star
	+ C_1 \, \alpha_\ell^{\k-1} y_\ell^\star
	+ C_1 \qctr \, \alpha_\ell^{\k-1} \beta_\ell^{\k-1}
	\big]\\
	&\quad
	+ \mu \, \big[
	\qctr \, \zeta_{\ell}(z_{\ell}^{\k-1}) \alpha_\ell^{\k-1}
	+ \zeta_{\ell}(z_{\ell}^{\k-1}) x_\ell^\star
	+ C_1 \, \beta_\ell^{\k-1} x_\ell^\star
	+ C_1 \qctr \, \alpha_\ell^{\k-1} \beta_\ell^{\k-1}
	\big]\\
	&\quad
	+ \mu^2 \, \big[
	q_\theta \, \eta_{\ell}(u_{\ell}^{\k-1}) \zeta_{\ell}(z_{\ell}^{\k-1})
	+ q_\theta C_1 \big(
	\eta_{\ell}(u_{\ell}^{\k-1}) \beta_{\ell}^{\k-1}
	+ \zeta_{\ell}(z_{\ell}^{\k-1}) \alpha_{\ell}^{\k-1}
	\big)
	+ C_1^2 \, \alpha_{\ell}^{\k-1} \beta_{\ell}^{\k-1}
	\big].
	\end{align*}
	Rearranging the terms, we obtain that
	\begin{equation}
	\label{eq:Lambda-estimate2}
	\begin{split}
	\Delta_{\ell+1}^0
	&\leq
	\big( \qctr^2 + 2 \mu \qctr C_1 + \mu^2 C_1^2 \big) \, \alpha_{\ell}^{\k-1} \beta_{\ell}^{\k-1}\\
	&\qquad
	+ \mu \, \big( \qctr + \mu q_\theta C_1 \big) \,
	\big[ \eta_{\ell}(u_{\ell}^{\k-1}) \beta_\ell^{\k-1} + \zeta_{\ell}(z_{\ell}^{\k-1}) \alpha_{\ell}^{\k-1} \big]\\
	&\qquad
	+ \mu^2 \, q_\theta \, \eta_{\ell}(u_{\ell}^{\k-1}) \zeta_{\ell}(z_{\ell}^{\k-1})
	+ R_\ell,
	\end{split}
	\end{equation}
	where the remainder term is defined as
	\begin{equation}
	\label{eq:remainder}
	R_\ell
	:=
	\mu \, \big[ \eta_{\ell}(u_{\ell}^{\k-1}) y_\ell^\star + \zeta_{\ell}(z_{\ell}^{\k-1}) x_\ell^\star \big]
	+ (\qctr + \mu C_1) \big[ \alpha_{\ell}^{\k-1} y_{\ell}^\star + \beta_{\ell}^{\k-1} x_{\ell}^\star \big]
	+ x_\ell^\star y_\ell^\star.
	\end{equation}
	Up to the choice of $\mu$, this concludes the proof.
\end{proof}

\begin{proof}[Proof of Lemma~\ref{lemma:banach2-new} (choosing $\mu$)]
	For Lemma~\ref{lemma:banach2-new}{\rm(i)}, we choose $\mu$ small enough such that $\qctr + 2 \mu \Cstab \leq 1$.
	From~\eqref{eq:Lambda-estimate1} and \eqref{eq:Lambda-estimate2} in the proofs of Lemma~\ref{lemma:banach2-new}{\rm(iii)--(iv)}, we see that we additionally require
	\begin{equation}
	\label{eq:mu-conditions}
	\qctr + \mu \lctr^{-1} (1+\qctr) < 1,
	\quad\quad
	\qctr^2 + 2 \mu \qctr C_1 + \mu^2 C_1^2 < 1,
	\quad\text{and}\quad
	\qctr + \mu q_\theta C_1 < 1.
	\end{equation}
	Choosing $\mu$ small enough, we satisfy all estimates.
	We define $\qaux < 1$ as the maximum of all terms in~\eqref{eq:mu-conditions} and $q_\theta$. 
\end{proof}

\begin{proof}[Proof of Lemma~\ref{lemma:banach2-new}\rm(v)]
	First, we note that from stability~\eqref{assumption:stab} it follows that 
	\begin{equation}
	\label{eq:Delta-estimator-relation}
		\eta_{\ell}(u_{\ell}^{\k-1})
		\lesssim
		\eta_{\ell}(u_{\ell}^{\star}) + \alpha_\ell^{\k-1}
		\quad\text{and}\quad
		\eta_{\ell}(u_{\ell}^{\star}) \zeta_{\ell}(z_{\ell}^{\star})
		\lesssim
		\Delta_{\ell}^{j}
		\text{ for all }
		0 \leq j \leq \k.
	\end{equation}
	Furthermore, Galerkin orthogonality and reliability~\eqref{assumption:rel} imply that, for all $n \in \N$ with $\ell'+n < \underline{\ell}$,
	\begin{equation}
	\label{eq:galerkin-y}
		\sum_{\ell=\ell'}^{\ell'+n} (y_\ell^\star)^2
		=
		\sum_{\ell=\ell'}^{\ell'+n} \enorm{z_{\ell+1}^\star - z_\ell^\star}^2\\
		=
		\enorm{z_{\ell'+n+1}^\star - z_{\ell'}^\star}^2
		\leq
		\enorm{z^\star - z_{\ell'}^\star}^2
		\stackrel{\mathclap{\eqref{assumption:rel}}}{\lesssim}
		\zeta_{\ell'}(z_{\ell'}^{\star})^2.
	\end{equation}
	With~\eqref{eq:Delta-estimator-relation} and~\eqref{eq:galerkin-y} for $n=1$, we can bound the remainder term from~\eqref{eq:remainder} by
	\begin{equation*}
	R_\ell
	\lesssim
	\eta_{\ell}(u_{\ell}^{\star}) y_\ell^\star + \zeta_{\ell}(z_{\ell}^{\star}) x_\ell^\star
	+ \alpha_\ell^{\k-1} y_\ell^\star + \beta_\ell^{\k-1} x_\ell^\star.
	\end{equation*}
	Next, let us recall from \cite[Lemma~3.6]{axioms} the quasi-monotonicity of the estimator, which follows from~\eqref{assumption:stab}--\eqref{assumption:rel} and the C\'ea lemma, i.e., for all $\ell' \leq \ell < \underline{\ell}$,
	\begin{equation}
	\label{eq:quasi-mon}
		\eta_\ell(u_\ell^\star)
		\leq
		\eta_{\ell'}(u_{\ell'}^\star) +  \Cstab \, \enorm{u_\ell^\star - u_{\ell'}^\star}
		\leq
		\eta_{\ell'}(u_{\ell'}^\star) +  \Cstab \, \enorm{u^\star - u_{\ell'}^\star}
		\lesssim
		\eta_{\ell'}(u_{\ell'}^\star).
	\end{equation}
	For $\eta_{\ell}(u_{\ell}^{\star}) y_\ell$, we get by summation for all $0 \leq j \leq \k(\ell')$ and all $n \in \N$ with $\ell'+n < \underline{\ell}$ that
	\begin{equation*}
		\sum_{\ell={\ell'}}^{{\ell'}+n} \eta_{\ell}(u_{\ell}^{\star})^2 (y_\ell^\star)^2
		\stackrel{\mathclap{\eqref{eq:quasi-mon}}}{\lesssim}
		\eta_{{\ell'}}(u_{{\ell'}}^{\star})^2 \sum_{\ell={\ell'}}^{{\ell'}+n} (y_\ell^\star)^2
		\stackrel{\mathclap{\eqref{eq:galerkin-y}}}{\lesssim}
		\eta_{{\ell'}}(u_{{\ell'}}^{\star})^2 \zeta_{{\ell'}}(z_{{\ell'}}^{\star})^2
		\stackrel{\mathclap{\eqref{eq:Delta-estimator-relation}}}{\lesssim}
		(\Delta_{{\ell'}}^j)^2.
	\end{equation*}
	Analogously, we see that
	\begin{equation}
	\label{eq:double-estimate}
		\sum_{\ell={\ell'}}^{{\ell'}+n} (x_\ell^\star)^2
		\lesssim
		\eta_{{\ell'}}(u_{{\ell'}}^{\star})^2
		\quad\text{as well as}\quad
		\sum_{\ell={\ell'}}^{{\ell'}+n} \zeta_{\ell}(z_{\ell}^{\star})^2 (x_\ell^\star)^2
		\lesssim
		(\Delta_{{\ell'}}^j)^2.
	\end{equation}
	We proceed with $\alpha_\ell^{\k-1} y_\ell^\star$. From~\eqref{eq:alpha-beta} and the Young inequality with $\delta > 0$, we see for $0 < \ell' \leq \ell < \underline{\ell}$ that
	\begin{equation*}
		(\alpha_\ell^{\k-1})^2
		\leq
		(\alpha_\ell^{0})^2
		\stackrel{\eqref{eq:alpha-beta}}{\leq}
		(1+\delta^{-1}) \, (x_{\ell-1}^\star)^2 + \qctr (1+\delta) \, (\alpha_{\ell-1}^{\k-1})^2.
	\end{equation*}
	For $\delta$ small enough such that $q_2 := \qctr(1+\delta) < 1$ and all for $0 \leq \ell \leq \ell' < \underline{\ell}$, the geometric series proves that
	\begin{equation*}
		(\alpha_{\ell}^{\k-1})^2
		\leq
		(1+\delta^{-1}) \sum_{j=\ell'}^{\ell-1} (x_{j}^\star)^2
		+ (\alpha_{\ell}^{\k-1})^2 \sum_{j=0}^{\infty} q_2^j
		\stackrel{\eqref{eq:double-estimate}}{\lesssim}
		\eta_{{\ell'}}(u_{{\ell'}}^{\star})^2 + (\alpha_{{\ell'}}^{\k-1})^2
	\end{equation*}
	and thus
	\begin{equation*}
		\sum_{\ell={\ell'}}^{{\ell'}+n} (\alpha_\ell^{\k-1})^2 (y_\ell^\star)^2
		\leq
		\big[ \eta_{{\ell'}}(u_{{\ell'}}^{\star})^2 + (\alpha_{{\ell'}}^{\k-1})^2 \big]
		\sum_{\ell={\ell'}}^{{\ell'}+n} (y_\ell^\star)^2
		\stackrel{\mathclap{\eqref{eq:galerkin-y}}}{\lesssim}
		\big[ \eta_{{\ell'}}(u_{{\ell'}}^{\star})^2 + (\alpha_{{\ell'}}^{\k-1})^2 \big]
		\zeta_{{\ell'}}(z_{{\ell'}}^{\star})^2
		\lesssim
		(\Delta_{{\ell'}}^{\k-1})^2.
	\end{equation*}
	Analogously, we see that $\sum_{\ell={\ell'}}^{{\ell'}+n} (\beta_\ell^{\k-1})^2 (x_\ell^\star)^2 \lesssim (\Delta_{{\ell'}}^{\k-1})^2$.
	Combining all estimates with
	\begin{equation*}
		R_\ell^2
		\lesssim
		\eta_{\ell}(u_{\ell}^{\star})^2 (y_\ell^\star)^2 + \zeta_{\ell}(z_{\ell}^{\star})^2 (x_\ell^\star)^2
		+ (\alpha_\ell^{\k-1})^2 (y_\ell^\star)^2 + (\beta_\ell^{\k-1})^2 (x_\ell^\star)^2,
	\end{equation*}
	we conclude the proof.
\end{proof}

With the foregoing auxiliary result, we are in the position to prove linear convergence.

\begin{proof}[Proof of Theorem~\ref{theorem:linear-convergence}]
	Let $(\ell, k) \in \QQ$.
	We recall the quasi-error products
	\begin{align*}
		\Lambda_\ell^k
		&=
		\big[ \, \enorm{u_{\ell}^{\star} - u_{\ell}^{k}} + \eta_\ell(u_{\ell}^{k})  \, \big]
		\big[ \, \enorm{z_{\ell}^{\star} - z_{\ell}^{k}} + \zeta_\ell(z_{\ell}^{k}) \, \big],\\
		\Delta_\ell^k
		&=
		\big[ \, \enorm{u_{\ell}^{\star} - u_{\ell}^{k}} + \mu \, \eta_\ell(u_{\ell}^{k})  \, \big]
		\big[ \, \enorm{z_{\ell}^{\star} - z_{\ell}^{k}} + \mu \, \zeta_\ell(z_{\ell}^{k}) \, \big] 
	\end{align*}
	from Theorem~\ref{theorem:linear-convergence} and Lemma~\ref{lemma:banach2-new}, respectively.
	Note that
	\begin{align*}
		\Lambda_\ell^k
		\leq
		\Delta_\ell^k
		\leq
		\mu^2 \, \Lambda_\ell^k
		\quad \textrm{if }
		\mu \geq 1,
		\qquad
		\Delta_\ell^k
		\leq
		\Lambda_\ell^k
		\leq
		\mu^{-2} \, \Delta_\ell^k
		\quad \textrm{if }
		\mu < 1,
	\end{align*}
	which yields the equivalence
	\begin{equation}
	\label{eq:delta-lambda}
		\min \{1, \mu^2 \} \, \Lambda_\ell^k
		\leq
		\Delta_\ell^k
		\leq
		\max \{1, \mu^2 \} \, \Lambda_\ell^k.
	\end{equation}

	We first show linear convergence of $\Delta_\ell^k$.
	By Lemma~\ref{lemma:banach2-new}{\rm(i)}, we can absorb the term $\Delta_{\ell'}^{\k}\le\Delta_{\ell'}^{\k-1}$ for all $\ell'$.
	Paying attention to the possible case $k=\k(\ell)$, this allows us to estimate
	\begin{equation*}
		\sum_{\substack{ (\ell',k') \in \QQ \\ |(\ell',k')| \geq |(\ell,k)| }}  (\Delta_{\ell'}^{k'})^2
		\lesssim
		(\Delta_\ell^k)^2 + 
		\sum_{k' = k}^{\k(\ell)-1} (\Delta_{\ell}^{k'})^2
		+ \sum_{\ell' = \ell+1}^{\underline{\ell}}  \sum_{k' = 0}^{\k(\ell')-1} (\Delta_{\ell'}^{k'})^2.
	\end{equation*}
	Lemma~\ref{lemma:banach2-new}{\rm(iii)} shows uniform reduction of the quasi-error on every level.
	This yields that
	\begin{equation*}
		\sum_{\substack{ (\ell',k') \in \QQ \\ |(\ell',k')| \geq |(\ell,k)| }} (\Delta_{\ell'}^{k'})^2
		\lesssim
		(\Delta_{\ell}^{k})^2 \sum_{k' = k}^{\k(\ell)} \qaux^{2(k'-k)}
		+ \sum_{\ell' = \ell+1}^{\underline{\ell}} (\Delta_{\ell'}^{0})^2 \sum_{k' = 0}^{\k(\ell')-1} \qaux^{2k'}
		\lesssim
		(\Delta_{\ell}^{k})^2 + \sum_{\ell' = \ell+1}^{\underline{\ell}} (\Delta_{\ell'}^{0})^2.
	\end{equation*}
	To estimate the sum over all levels, we use that, for the refinement step, Lemma~\ref{lemma:banach2-new}{\rm(iv)} shows contraction up to a remainder term.
	The Young inequality with $\delta > 0$ and Lemma~\ref{lemma:banach2-new}{\rm(i)} then prove that
	\begin{align*}
		(\Delta_{\ell'}^{0})^2
		&\leq
		\qaux^2 (1+\delta) \, (\Delta_{\ell'-1}^{\k-1})^2 + (1+\delta^{-1}) \, R_{\ell'-1}^2\\
		&\leq
		\qaux^2 (1+\delta) \, (\Delta_{\ell'-1}^{0})^2 + (1+\delta^{-1}) \, R_{\ell'-1}^2
		\quad \text{for all }
		0 < \ell' \leq \underline{\ell}.
	\end{align*}
	Choosing $\delta$ small enough such that $q := \qaux^2 (1+\delta) < 1$, we obtain from repeatedly applying the previous estimates that
	\begin{equation*}
		(\Delta_{\ell'}^{0})^2
		\leq
		q^{\ell'-\ell} \, (\Delta_{\ell}^{\k-1})^2
		+ (1+\delta^{-1}) \sum_{n = \ell}^{\ell'-1} q^{(\ell'-1)-n} \, R_{n}^2
		\quad \text{for all }
		0 \leq \ell < \ell' \leq \underline{\ell}.
	\end{equation*}
	Using this estimate and a change of summation indices, the geometric series and Lemma~\ref{lemma:banach2-new}{\rm(v)} uniformly bound the sum over all levels by
	\begin{align*}
		\sum_{\ell' = \ell+1}^{\underline{\ell}} (\Delta_{\ell'}^{0})^2
		&\lesssim
		\sum_{\ell' = \ell+1}^{\underline{\ell}}
		\Big[
		q^{\ell'-\ell} \, (\Delta_{\ell}^{\k-1})^2
		+ \sum_{n = \ell}^{\ell'-1} q^{(\ell'-1)-n} \, R_{n}^2
		\Big]\\
		&\lesssim
		(\Delta_{\ell}^{\k-1})^2
		+\sum_{n=\ell}^{\underline{\ell}-1} R_{n}^2 \sum_{i = 0}^{\infty} q^{i}
		\lesssim
		(\Delta_{\ell}^{\k-1})^2
		+\sum_{n=\ell}^{\underline{\ell}-1} R_{n}^2
		\stackrel{\rm(v)}{\lesssim}
		(\Delta_{\ell}^{\k-1})^2.
	\end{align*}
	Combining the estimates above, we obtain that
	\begin{equation*}
		\sum_{\substack{ (\ell',k') \in \QQ \\ |(\ell',k')| \geq |(\ell,k)| }} (\Delta_{\ell'}^{k'})^2
		\lesssim
		(\Delta_{\ell}^{k})^2 + \sum_{\ell' = \ell+1}^{\underline{\ell}} (\Delta_{\ell'}^{0})^2
		\lesssim
	(\Delta_{\ell}^{k})^2 + (\Delta_{\ell}^{\k-1})^2.
	\end{equation*}
	In the case $k < \k(\ell)$, Lemma~\ref{lemma:banach2-new}{\rm(i)} proves that
	\begin{equation*}
	\sum_{\substack{ (\ell',k') \in \QQ \\ |(\ell',k')| \geq |(\ell,k)| }} (\Delta_{\ell'}^{k'})^2
	\leq
		C \, (\Delta_{\ell}^{k})^2.
	\end{equation*}
	In the case $k = \k(\ell)$, this follows with Lemma~\ref{lemma:banach2-new}{\rm(ii)}.
	In either case, the constant $C > 0$ depends only on $\Caux$ and $\qaux$ from Lemma~\ref{lemma:banach2-new}.
	Basic calculus then provides the existence of $\Clin^\prime := (1+C)^{1/2} > 1$ and $0 < \qlin := (1-C^{-1})^{-1/2} < 1$ such that
	\begin{equation*}
		\Delta_{\ell'}^{k'}
		\leq
		\Clin^\prime \qlin^{|(\ell',k')| - |(\ell,k)|} \, \Delta_{\ell}^{k}
		\quad
		\text{for all } (\ell,k), (\ell',k') \in \QQ \text{ with } (\ell',k') \geq (\ell,k);
	\end{equation*}
	see~\cite[Lemma~4.9]{axioms}.
	Finally, the claim of Theorem~\ref{theorem:linear-convergence} follows from~\eqref{eq:delta-lambda} with $\Clin = \max \{ \mu^{-2}, \mu^2 \} \, \Clin^\prime$.
\end{proof}

%\clearpage
%!TEX root = main.tex

%%%%%%%%%%%%%%%%%%%%%%%%%%%%%%%%%%%%%%%%%%%%%%%%%%%%%%%%%%%%%%%%%%%%%%%%%%%%%%%%%%%
%%%%%%%%%%%%%%%%%%%%%%%%%%%%%%%%%%%%%%%%%%%%%%%%%%%%%%%%%%%%%%%%%%%%%%%%%%%%%%%%%%%
\section{Proof of Theorem~\ref{theorem:optimal-convergence} (optimal rates)}\label{sec:proof2}
%%%%%%%%%%%%%%%%%%%%%%%%%%%%%%%%%%%%%%%%%%%%%%%%%%%%%%%%%%%%%%%%%%%%%%%%%%%%%%%%%%%
%%%%%%%%%%%%%%%%%%%%%%%%%%%%%%%%%%%%%%%%%%%%%%%%%%%%%%%%%%%%%%%%%%%%%%%%%%%%%%%%%%%

We recall the following comparison lemma from~\cite{pointabem}. While~\cite{pointabem} is concerned with point errors in boundary element computations, we stress that the proof of~\cite[Lemma~14]{pointabem} works on a completely abstract level and thus is applicable here as well.

\begin{lemma}[{{\cite[Lemma~14]{pointabem}}}]\label{lemma:comparison}
    The overlay estimate~\eqref{eq:mesh-overlay} and the axioms~\eqref{assumption:stab}--\eqref{assumption:red} and \eqref{assumption:drel} yield the existence of a constant $C_1>0$ such that, given $0<\kappa<1$, each mesh $\TT_\coarse\in\T$ admits some refinement $\TT_\fine\in\T(\TT_\coarse)$ such that for all $s,t>0$, it holds that 
    \begin{subequations}
    \begin{align}
        \label{eq1:comparison}
        \eta_\fine(u_\fine^\star)^2 \zeta_\fine(z_\fine^\star)^2 
        &\le \kappa^2 \eta_\coarse(u_\coarse^\star)^2 \zeta_\coarse(z_\coarse^\star)^2, 
        \\
        \label{eq2:comparison}
        \#\TT_\fine - \#\TT_\coarse 
        &\le 2 \big(C_1 \kappa^{-1} \norm{u^\star}{\mathbb{A}_s} \norm{z^\star}{\mathbb{A}_t}\big)^{1/(s+t)} 
        \big(\eta_\coarse(u_\coarse^\star)\zeta(z_\coarse^\star) \big)^{1/(s+t)}.
    \end{align}
    \end{subequations}
    The constant $C_1$ depends only on $\Cstab$, $\qred$, and $\Cdrel$. 
    \hfill$\square$
\end{lemma}

Note that \eqref{eq1:comparison} immediately implies that 
\begin{align}\label{eq:either-reduction}
    \eta_\fine(u_\fine)^2 \le \kappa \eta_\coarse(u_\coarse^\star)^2
    \quad\text{or}\quad
    \zeta_\fine(z_\fine^\star)^2 \le \kappa \zeta_\coarse(z_\coarse^\star)^2.
\end{align}
We will employ this lemma in combination with the so-called optimality of D\"orfler marking from \cite{axioms}.

\begin{lemma}[{{\cite[Proposition~4.12]{axioms}}}]\label{lemma:optimality-doerfler}
    Under \eqref{assumption:stab} and \eqref{assumption:drel}, for all $0<\Theta'<1/(1+\Cstab^2\Cdrel^2)$, there exists $0<\kappa_{\Theta'}<1$ such that for all $\TT_\coarse\in\T$ and all $\TT_\fine\in\T(\TT_\coarse)$, \eqref{eq:either-reduction} with $\kappa=\kappa_{\Theta'}$ implies that 
    \begin{align}\label{eq:optimal-doerfler}
        \Theta' \eta_\coarse(u_\coarse^\star)^2 \le \eta_\coarse(\TT_\coarse\setminus\TT_\fine,u_\coarse^\star)^2
        \quad\text{or}\quad
        \Theta' \zeta_\coarse(z_\coarse^\star)^2 \le \zeta_\coarse(\TT_\coarse\setminus\TT_\fine,z_\coarse^\star)^2. 
    \end{align}
    The constant $\kappa_{\Theta'}$ depends only on $\Cstab$, $\Cdrel$, and $\Theta'$.
    \hfill$\square$    
\end{lemma}

The next lemma is already implicitly found in~\cite{banach}.
It shows that, if $\lctr > 0$ is sufficiently small, then D\"orfler marking for the exact discrete solution implicitly implies D\"orfler marking for the approximate discrete solution.
This will turn out to be the key observation to prove optimal convergence rates.
We include the proof for the convenience of the reader.

\begin{lemma}\label{lemma:banach}
	Suppose \eqref{assumption:stab}--\eqref{assumption:rel}.
	Let $0< \Theta \le 1$ and $0 < \lctr < \lconv:=(1-\qctr) / ( \qctr \Cstab)$.
	Define $\Theta' := \big(\frac{\sqrt{\Theta} + \lctr/\lconv}{1-\lctr/\lconv}\big)^2$.
	Then, as soon as the iterative solver terminates~\eqref{eq:stopping}, there hold the following statements~{\rm(i)--(iv)} for all $0 \leq \ell < \underline{\ell}$ and all $\UU_\ell\subseteq \TT_\ell$:
	\begin{itemize}
		\item[\rm(i)]
		$(1 - \lctr/\lconv) \, \eta_\ell(u_{\ell}^{\m}) \le \eta_\ell(u_\ell^\star)
		\le (1 + \lctr/\lconv) \, \eta_\ell(u_{\ell}^{\m})$.
		
		\item[\rm(ii)]
		$\Theta \, \eta_\ell(u_\ell^\m)^2
		\le \eta_\ell(\UU_\ell, u_\ell^\m)^2$ \quad
		provided that
		$\Theta' \, \eta_\ell(u_{\ell}^\star)^2
		\le \eta_\ell(\UU_\ell, u_{\ell}^\star)^2$.
		
		\item[\rm(iii)]
		$(1 - \lctr/\lconv) \, \zeta_\ell(z_{\ell}^{\n})
		\le \zeta_\ell(z_\ell^\star) \le (1 + \lctr/\lconv) \, \zeta_\ell(z_{\ell}^{\n})$.
		
		\item[\rm(iv)]
		$\Theta \, \zeta_\ell(z_\ell^\n)
		\le \zeta_\ell(\UU_\ell, z_\ell^\n)$ \quad
		provided that
		$\Theta' \, \zeta_\ell(z_{\ell}^\star)^2
		\le \zeta_\ell(\UU_\ell, z_{\ell}^\star)^2$.
	\end{itemize}
\end{lemma}

\begin{proof}
	It holds that
	\begin{align*}
	&\eta_\ell(\UU_\ell, u_\ell^\star)
	\reff{assumption:stab}\le \eta_\ell(\UU_\ell, u_{\ell}^{\m}) + \Cstab \, \enorm{u_\ell^\star - u_{\ell}^{\m}}
	\reff{eq:contraction}\le \eta_\ell(\UU_\ell, u_{\ell}^{\m}) + \Cstab \, \frac{\qctr}{1-\qctr} \, \enorm{u_{\ell}^{\m} - u_{\ell}^{ \m-1}}
	\\& \quad
	\reff{eq:stopping}\le \eta_\ell(\UU_\ell, u_{\ell}^{\m}) + \Cstab \, \frac{\qctr}{1-\qctr} \, \lctr \, \eta_\ell(u_{\ell}^{\m})
	\refp{eq:contraction}= \eta_\ell(\UU_\ell, u_{\ell}^{\m}) + \frac{\lctr}{\lconv} \, \eta_\ell(u_{\ell}^{\m}).
	\end{align*}
	The same argument proves that
	\begin{align*}
	\eta_\ell(\UU_\ell, u_{\ell}^{\m}) \le \eta_\ell(\UU_\ell, u_\ell^\star) + \frac{\lctr}{\lconv} \, \eta_\ell(u_{\ell}^{\m}).
	\end{align*}
	For $\UU_\ell = \TT_\ell$, the latter two estimates lead to
	\begin{align*}
	(1 - \lctr/\lconv) \, \eta_\ell(u_{\ell}^{\m}) \le \eta_\ell(u_\ell^\star) \le (1 + \lctr/\lconv) \, \eta_\ell(u_{\ell}^{\m}).
	\end{align*}
	This concludes the proof of~(i). To see~(ii), we use the assumption
	\begin{align*}
	(1-\lctr/\lconv) \, \sqrt{\Theta'} \, \eta_\ell(u_\ell^{\m})
	\stackrel{\rm(i)}\le \sqrt{\Theta'} \, \eta_\ell(u_{\ell}^\star)
	\le \eta_\ell(\UU_\ell, u_{\ell}^\star)
	\le \eta_\ell(\UU_\ell, u_\ell^\m) + \frac{\lctr}{\lconv} \, \eta_\ell(u_{\ell}^{\m}).
	\end{align*}
	Noting that $\sqrt{\Theta} = (1-\lctr/\lconv)\,\sqrt{\Theta'} - \lctr/\lconv$,
	this concludes the proof of~(ii). The remaining claims~(iii)--(iv) follow verbatim. 
\end{proof}

\begin{proof}[Proof of Theorem~\ref{theorem:optimal-convergence}]
    By Corollary~\ref{corollary:optimal-convergence}, it is sufficient to prove that
    \begin{align*}
        C_{s+t}
        =
        \sup_{(\ell,k) \in \QQ} \big( \#\TT_\ell - \#\TT_0 + 1 \big)^{s+t} \Lambda_\ell^k
        \lesssim
        \max\{\norm{u^\star}{\mathbb{A}_s}\norm{z^\star}{\mathbb{A}_t},\Lambda_0^0\}.
    \end{align*}
    We prove this inequality in two steps.
    
    \textbf{Step 1:} In this step, we bound the number of marked elements $\#\MM_{\ell'}$ for arbitrary $0\le \ell'<\underline{\ell}$.
    Let $\Theta > 0$ and corresponding $\Theta'$ from Lemma~\ref{lemma:banach} such that
    \begin{equation}
    \label{eq:generic-theta}
	    \Theta' = \Big( \frac{\sqrt{\Theta}+\lctr/\lconv}{1-\lctr/\lconv} \Big)^2
	    < \frac{1}{1+\Cstab^2\Cdrel^2}.
    \end{equation}
	Let $\TT_{\fine(\ell')}\in\T(\TT_{\ell'})$ be the corresponding mesh as in Lemma~\ref{lemma:comparison}. 
	With Lemma~\ref{lemma:optimality-doerfler}, this yields that 
	\begin{align*}
		\Theta'\eta_{\ell'}(u_{\ell'}^\star)^2 \le \eta_{\ell'}(\TT_{\ell'}\setminus\TT_{\fine(\ell')},u_{\ell'}^\star)^2
		\quad\text{or}\quad
		\Theta'\zeta_{\ell'}(z_{\ell'}^\star)^2 \le \zeta_{\ell'}(\TT_{\ell'}\setminus\TT_{\fine(\ell')},z_{\ell'}^\star)^2. 
	\end{align*}
	Lemma~\ref{lemma:banach} with $\UU_{\ell'} = \TT_{\ell'}\setminus\TT_{\fine(\ell')}$ shows that
	\begin{equation}
	\label{eq:generic-doerfler}
		\Theta\eta_{\ell'}(u_{\ell'}^\m)^2 \le \eta_{\ell'}(\TT_{\ell'}\setminus\TT_{\fine(\ell')},u_{\ell'}^\star)^2
		\quad\text{or}\quad
		\Theta\zeta_{\ell'}(z_{\ell'}^\n)^2 \le \zeta_{\ell'}(\TT_{\ell'}\setminus\TT_{\fine(\ell')},z_{\ell'}^\star)^2. 
	\end{equation}
	We consider the marking strategies from Remark~\ref{rem:marking} separately.

	For strategy {\rm (a)}, we have with $\Theta := 2\theta$ and assumption~\eqref{eq:small-theta} that~\eqref{eq:generic-theta} is satisfied.
	Hence, \eqref{eq:generic-doerfler} implies that there holds~\eqref{eq:abstract-marking}, i.e., 
	\begin{align*}
		2 \theta \eta_{\ell'}(u_{\ell'}^\m)^2 \zeta_{\ell'}(z_{\ell'}^\n)^2 
		\le \eta_{\ell'}(\TT_{\ell'}\setminus\TT_{\fine(\ell')},u_{\ell'}^\m)^2 \zeta_{\ell'}(z_{\ell'}^\n)^2 
		+ \eta_{\ell'}(u_{\ell'}^\m)^2 \zeta_{\ell'}(\TT_{\ell'}\setminus\TT_{\fine(\ell')},z_{\ell'}^\n)^2.
	\end{align*}
	By assumption of Theorem~\ref{theorem:optimal-convergence}, $\MM_{\ell'}$ is essentially minimal with~\eqref{eq:abstract-marking}. 
	We infer that
	\begin{equation}
	\label{eq:M-estimate}
		\#\MM_{\ell'} 
		\leq
		\Cmark \#(\TT_{\ell'}\setminus\TT_{\fine(\ell')}) 
		\stackrel{\eqref{eq:mesh-sons}}{\lesssim}
		\#\TT_{\fine(\ell')} - \#\TT_{\ell'}. 
	\end{equation}
	For the strategies {\rm (b)--(c)}, we set $\Theta = \theta$ and note that assumption~\eqref{eq:small-theta} (as well as the weaker assumption~\eqref{eq:larger-theta}) imply~\eqref{eq:generic-theta}, and hence~\eqref{eq:generic-doerfler}.
	Again, by assumption of Theorem~\ref{theorem:optimal-convergence}, $\MM_\ell$ is chosen essentially minimal (with an additional factor two for the strategy {\rm (c)}) such that~\eqref{eq:generic-doerfler} holds.
	For all three strategies, we therefore conclude that
    \begin{align*}
        \#\MM_{\ell'} 
        \lesssim \#\TT_{\fine(\ell')} - \#\TT_{\ell'}
        &\stackrel{\mathclap{\eqref{eq2:comparison}}}\lesssim ~~~
        \big(\norm{u^\star}{\mathbb{A}_s}\norm{z^\star}{\mathbb{A}_t}\big)^{1/(s+t)} \big(\eta_{\ell'}(u_{\ell'}^\star)\zeta_{\ell'}(z_{\ell'}^\star)\big)^{-1/(s+t)}
        \\
        &\stackrel{\mathclap{\text{Lem.\ref{lemma:banach}}}}\lesssim ~~~
        \big(\norm{u^\star}{\mathbb{A}_s}\norm{z^\star}{\mathbb{A}_t}\big)^{1/(s+t)} \big(\eta_{\ell'}(u_{\ell'}^\m)\zeta_{\ell'}(z_{\ell'}^\n)\big)^{-1/(s+t)}.
    \end{align*}
    Recall that \eqref{eq:contraction} and \eqref{eq:stopping} give that $\eta_{\ell'}(u_{\ell'}^\k)\zeta_{\ell'}(z_{\ell'}^\k)\simeq\Lambda_{\ell'}^\k$. 
    This finally shows that
    \begin{align*}
        \#\MM_{\ell'} \lesssim \big(\norm{u^\star}{\mathbb{A}_s}\norm{z^\star}{\mathbb{A}_t}\big)^{1/(s+t)} (\Lambda_{\ell'}^\k)^{-1/(s+t)}. 
    \end{align*}
    
    \textbf{Step 2:}
    Let $(\ell,k)\in\QQ$. First, we consider $\ell>0$ and thus $\#\TT_\ell>\#\TT_0$.
    The closure estimate and Step~1 prove that 
    \begin{align*}
        \#\TT_\ell - \#\TT_0 + 1 
        \simeq \#\TT_\ell -\#\TT_0 
        \stackrel{\eqref{eq:mesh-closure}}\lesssim \sum_{\ell'=0}^{\ell-1} \#\MM_{\ell'}
        \lesssim \big(\norm{u^\star}{\mathbb{A}_s}\norm{z^\star}{\mathbb{A}_t}\big)^{1/(s+t)} 
        \sum_{\ell'=0}^{\ell-1}(\Lambda_{\ell'}^\k)^{-1/(s+t)}
        \\
        \le \big(\norm{u^\star}{\mathbb{A}_s}\norm{z^\star}{\mathbb{A}_t}\big)^{1/(s+t)} 
        \sum_{\substack{(\ell',k')\in\QQ\\|(\ell',k')|\ge|(\ell,k)|}}(\Lambda_{\ell'}^{k'})^{-1/(s+t)}.
    \end{align*}
    Linear convergence of Theorem~\ref{theorem:linear-convergence}, further shows that 
    \begin{align*}
        \#\TT_\ell-\#\TT_0+1 
        &\lesssim \big(\norm{u^\star}{\mathbb{A}_s}\norm{z^\star}{\mathbb{A}_t}\big)^{1/(s+t)} \Clin^{1/(s+t)}  (\Lambda_{\ell}^k)^{-1/(s+t)}
        \sum_{\substack{(\ell',k')\in\QQ\\|(\ell',k')|\ge|(\ell,k)|}} (\qlin^{1/(s+t)})^{|(\ell,k)|-|(\ell',k')|}
        \\
        &\le \big(\norm{u^\star}{\mathbb{A}_s}\norm{z^\star}{\mathbb{A}_t}\big)^{1/(s+t)} \frac{\Clin^{1/(s+t)}}{1-\qlin^{1/(s+t)}}\, \Clin^{1/(s+t)}  (\Lambda_{\ell}^k)^{-1/(s+t)}.
    \end{align*}
    Rearranging this estimate, we see that
    \begin{align*}
        (\#\TT_\ell-\#\TT_0+1)^{s+t}\Lambda_\ell^k \lesssim \norm{u^\star}{\mathbb{A}_s}\norm{z^\star}{\mathbb{A}_t}
        \quad\text{for all }(\ell,k)\in\QQ\text{ with }\ell>0.
    \end{align*}
    It remains to consider $\ell=0$.
    By Theorem~\ref{theorem:linear-convergence}, we have that 
    \begin{align*}
        (\#\TT_\ell - \#\TT_0 + 1)^{s+t} \Lambda_\ell^k = \Lambda_0^k \lesssim \Lambda_0^0
        \quad\text{for all }(\ell,k)\in\QQ\text{ with }\ell=0.
    \end{align*}
    This concludes the proof.
\end{proof}

%\clearpage

%%%%%%%%%%%%%%%%%%%%%%%%%%%%%%%%%%%%%%%%%%%%%%%%%%%%%%%%%%%%%%%%%%%%%%%%%%%%%%%%%%%
%%%%%%%%%%%%%%%%%%%%%%%%%%%%%%%%%%%%%%%%%%%%%%%%%%%%%%%%%%%%%%%%%%%%%%%%%%%%%%%%%%%

%\renewcommand{\section}[3][]{\vskip4mm\begin{center}\bf\normalsize R\small EFERENCES\normalsize\end{center}\vskip2mm}

\bibliographystyle{alpha}
\bibliography{literature}

\end{document}